\documentclass[11 pt]{article}
\usepackage{graphicx,tikz,tikz-cd,enumerate,amssymb,amsmath,amsthm}
\usetikzlibrary{matrix,arrows,decorations.pathmorphing}
\newcommand{\Hecke}{\mathcal{H}}
\newcommand{\Hom}{\mathrm{Hom}}
\newcommand{\id}{\mathrm{Id}}

\newcommand{\ldes}{\mathcal{L}}
\newcommand{\tif}{\text{ if }}
\newcommand{\tand}{\text{ and }}
\newcommand{\T}{T}
\newcommand{\C}{\tilde{C}}

\newcommand{\Tau}{\mathcal{T}}
\newcommand{\trho}{\tilde{\rho}}

\newcommand{\tphi}{\tilde{\phi}}
\newcommand{\tTheta}{\tilde{\Theta}}

\newcommand{\cm}[3]{\sum\limits_{j=0}^{\lfloor \frac{#1-1}{2} \rfloor}  (-1)^j \binom{#1-1-j}{j} ( \dots #3 #2 #3)_{#1-2j}}
\newtheorem{theorem}{Theorem}[section]
\newtheorem{lemma}[theorem]{Lemma}
\newtheorem{proposition}[theorem]{Proposition}
\newtheorem{corollary}[theorem]{Corollary}

\newtheorem{remark}[theorem]{Remark}
\newtheorem{definition}[theorem]{Definition}
 \newtheorem{example}[theorem]{Example}
 \numberwithin{equation}{section}
\begin{document}

\markboth{Alexander Diaz-Lopez}
{Representations of Hecke algebras on quotients of path algebras}

%\catchline{}{}{}{}{}
%
%%%%%%%%%%%%%%%%%%%%%%%%%%%%%%%%%%%%%%%%%%%%%%%%%%%%%%%%%%%%%%%%%%%%

\title{Representations of Hecke algebras on quotients of path algebras
}

\author{Alexander Diaz-Lopez}

%\address{Department of Mathematics, University of Notre Dame, 255 Hurley Hall\\
%Notre Dame, IN 46566,USA\,\\
%\email{adiaz4@nd.edu} }

\maketitle

%\begin{history}
%\received{(Day Month Year)}
%\accepted{(Day Month Year)}
%\comby{[editor]}
%\end{history}

\begin{abstract}
Let $(W,S)$ be a Coxeter system. A $W$-graph encodes a representation of the Hecke algebra $\mathcal{H}$ of $W$. We construct universal representations of multi-parameter Hecke algebras on certain quotients of path algebras, and study their relationships with $W$-graph representations. We also study the quotients of path algebras on their own, motivated by one example where the quotient path algebra is isomorphic to an ideal of Lusztig asymptotic Hecke algebra. Finally, we describe a method to obtain a generating set for the ideals by which we quotient the path algebras.
\end{abstract}

%\keywords{Hecke algebras; W-graphs; Coxeter groups; representation theory, path algebras.}

%\ccode{Mathematics Subject Classification 2010: 20C08}

\section{Introduction}

Given a Coxeter system $(W,S)$, a $W$-graph \cite{KL1} is a combinatorial object that encodes a representation of the Hecke algebra $\mathcal{H}$ of $W$ over a commutative unital ring $A$. In one version, a $W$-graph consists of a quiver (directed graph) with vertices labeled by subsets of $S$, and edges labeled by elements of $A$ subject to stringent conditions. These conditions ensure that certain endomorphisms (defined in terms of this data) of the free $A$-module on the set of vertices give rise to a representation of $\mathcal{H}$. Recent developments on $W$-graphs include \cite{HN-WGI, HY-IWG,HY-IWG2, N-WGI2, JS,JS-AFTFWG,YIN-IWGFSG}.

In this paper we consider the relationships between representations (denoted $\tau$) attached to $W$-graphs for multi-parameter Hecke algebras  and universal representations (denoted $\trho$)  on quotients of path algebras attached to arbitrary quivers (called pre-$D$-graphs) with vertices labeled by subsets of $S$, subject to only mild finiteness conditions. The precise meaning of ``universal" is given in Section \ref{sec:univ}. Given a pre-$D$-graph, the universal representation $\trho$ acts on, roughly, the largest quotient of the path algebra on which the analogue of the formulae defining $W$-graph representations affords an actual representation. An important subtlety is that in the construction, $A$ is a free module over a subalgebra $A_0$, and the construction is done so that the representation is defined on $\mathcal{R}:=A\otimes_{A_0}\mathcal{R}_0$ where $\mathcal{R}_0$ is an $A_0$-algebra (and a quotient of a path algebra) produced by the construction. We construct an equivariant map between the representations $\tau$ and $\trho$ (Theorem \ref{thm:eqmap}) and state a more precise result (Theorem \ref{thm:eqmapmorita}) closely relating $\trho$ to a representation Morita equivalent to $\tau$.

We also study the quotient path algebras $\mathcal{R}$ on their own. The main motivation is given by a special example (Example \ref{subsec:AHA}) in which the quotient path algebra is isomorphic to the two-sided ideal of the asymptotic Hecke algebra associated to the two-sided cells of $a$-value 1 (see \cite{L1}, Chapter 18). This example suggest a stronger connection (in general) between quotient path algebras and the asymptotic Hecke algebra. We show three formal properties of $\mathcal{R}$ which are analogous to known (non-formal) properties of the asymptotic Hecke algebra and special representations of Weyl groups. First, given two specializations $\mathcal{H}', \mathcal{H}''$ of $\mathcal{H}$ to algebras over a commutative unital ring $B$, the construction gives rise to a canonical $(\mathcal{H}',\mathcal{H}'')$-bimodule structure on $B\otimes_{A_0}\mathcal{R}_0$ (Theorem \ref{thm:bimod}). Secondly, given a pre-$D$-graph subject to mild finiteness conditions, the opposite ring of $\mathcal{R}$ arises by applying the same construction to a dual pre-$D$-graph (Theorem \ref{thm:Rdual}). Thirdly, under suitable (strong) conditions, when the pre-$D$-graph comes from a $W$-graph, the $W$-graph representation $\tau$ appears at least once in each left ideal of $\mathcal{R}$ generated by an idempotent corresponding to a vertex of the quiver (Corollary \ref{cor:Uxsurjective}).

As is the case for the asymptotic Hecke algebra, it is hard to make computations or give an explicit description of these quotient path algebras. However, we describe a method to obtain a generating set for the ideal by which we quotient the path algebra. We first compute the generators of the ideals associated to relatively simple pre-$D$-graphs (called universal pre-$D$-graphs) and their images under a certain homomorphism give a generating set for the original ideal (Theorem \ref{thm:psimap}).

The paper is organized in the following way.
 In Section \ref{sec:prelim} we discuss notation and some known properties, mostly on path algebras. 
In Section \ref{sec:IHAD} we extend the notion of $W$-graphs to ``pre-$D$-graphs" and ``$D$-graphs" (which are essentially $W$-graphs for multi-parameter Hecke algebras).
In Section \ref{sec:QPA} we discuss the main topic of this paper, quotient path algebras, and construct representations of Hecke algebras on them.
In Section \ref{sec:EandR} we construct an equivariant map between the representations of Hecke algebras coming from $D$-graphs and those defined on quotient path algebras. 
In Section \ref{universalPDG} we describe a method to compute a generating set for the ideals by which we quotient the path algebras. We then discuss three interesting examples of this theory. In Section 7 we show how the objects of interests in this paper appear as examples of a more general notion of $W$-graphs, namely ``$D$-graphs over non-commutative algebras".
	%%%%%%%%%%%%
	%%%%%%%%%%%%
	%%%%%%%%%%%%
\section{Preliminaries}\label{sec:prelim}
Throughout this paper a ring is not required to have an identity and all modules are left modules unless otherwise noted. Let $A$ be a commutative unital ring. An $A$-algebra $R$ is a ring together with a unitary $A$-module structure such that the multiplication map $R\times R \rightarrow R$ is $A$-bilinear. An $A$-algebra homomorphism is defined as an $A$-linear ring homomorphism. A module $M$ over the $A$-algebra $R$ is a module over the ring $R$ together with an $A$-module structure such that for $r\in R, a\in A, m\in M$ we have $a(rm)=(ar)m=r(am)$. A module homomorphism over the $A$-algebra $R$ is defined as a module homomorphism over the ring $R$ which is also $A$-linear (i.e., for an $R$-module homomorphism $\theta:N\rightarrow M$ we have $\theta(an+n')=a\theta(n)+\theta(n')$ for all $a \in A, n,n'\in N$). Let $End_R(R)$ be the unital $A$-algebra consisting of $R$-module endomorphisms of $R$ under addition and composition with $A$-module structure given by $(a \theta) (r) := a\theta(r) \in R$ for all $a \in A, r \in R$. By the definitions above we have $End_RR\subseteq End_A R$.  We now introduce some notation and present three known results (without proofs) to be used in subsequent sections.
	%%%%%%%%%%%%
	\subsection{Rings with enough orthogonal idempotents}
		A ring $R$ has enough orthogonal idempotents if there exists a set of orthogonal idempotent elements $\{e_x\}_{x\in I}$ (i.e., $e_x e_x=e_x$ and $e_xe_y=0$ for $x\neq y$) such that $R=\oplus_{x \in I} e_x R = \oplus_{x\in I} R e_x$, where $I$ is some index set. Given a commutative unital ring $A$, we say an $A$-algebra $R$ is an $A$-algebra with enough orthogonal idempotents if it has enough orthogonal idempotents as a ring.
	\begin{proposition}\label{prop:quotientidemp}
		Let $R$ be an $A$-algebra with enough orthogonal idempotents given by $\{ e_x\}_{x\in I}$, then the following statements hold.
			\begin{enumerate}[(I)]
				\item Let $J$ be an ideal of $R$ then $R/J$ is an $A$-algebra with enough orthogonal idempotents given by $\{[e_x]\}_{x\in I}$, where $[e_x]$ is the image of $e_x$ in $R/J$.\label{property:quotientidemp}
				\item $End_R(R)\cong \prod_{x\in I} e_x R$ as $A$-algebras, where the addition on $\prod_{x\in I} e_x R$ is component-wise, the $A$-action is defined as $a\cdot (c_x)_{x\in I}:=(ac_x)_{x\in I}$ and multiplication is given by $(c_x)_{x\in I}\cdot (d_y)_{y\in I}:= (\sum_{x \in I} d_yc_x)_{y\in I}$. An isomorphism is given by $\theta \mapsto (\theta(e_x))_{x\in I}$. \label{property:isoend}
			\end{enumerate}
	\end{proposition}
	%%%%%%%%%%%%
	\subsection{Quivers and path algebras}
	A quiver is a collection of vertices and oriented edges, with loops and parallel arrows allowed. More formally, a quiver is a $4$-tuple $\Gamma=(V^{\Gamma},E^{\Gamma},s,t)$ where $V^{\Gamma}, E^{\Gamma}$ are two sets, and $s,t: E^{\Gamma} \rightarrow V^{\Gamma}$ are two set-theoretic maps. The elements of $V^{\Gamma}$ are called vertices, the elements of $E^{\Gamma}$ are called edges, and for any $e \in E^{\Gamma}$ the elements $s(e), t(e) \in V^{\Gamma}$ are called the source and target of $e$, respectively. We denote an element $e \in E^{\Gamma}$ as $t(e) \xleftarrow{e} s(e)$. Given any quiver $\Gamma$, we define a path $p$ as a (possibly empty) finite sequence of edges $(e_1,\dots, e_n)$ such that $s(e_i)=t(e_{i+1})$ for all appropriate $i$. If the sequence is non-empty, we say the length of $p$ is $n$, the source of $p$ is $s(e_n)$, and the target of $p$ is $t(e_1)$. If the sequence is empty we must choose an element in $V^{\Gamma}$ to be both the source and target of $p$. We call these elements paths of length zero and identify the set of paths of length zero with $V^{\Gamma}$. Let $P^{\Gamma}$ be the set of all paths. When convenient, we denote a path $p=(e_1, \dots, e_n) \in P^{\Gamma}$ as $x_0 \xleftarrow{e_1} x_1 \xleftarrow{e_2}\cdots \xleftarrow{e_n}x_n$, where $e_i$ is the edge $x_{i-1}\xleftarrow{e_i}x_i$ for $i\in \{1,2,\dots,n\}$.
	
	We define the path algebra of $\Gamma$ over a commutative unital ring $A$ as follows. Let $A[\Gamma]$ be the free $A$-module on the set $P^{\Gamma}$. Given $p=(e_1,\dots,e_n),q=(f_1\dots, f_m) \in P^{\Gamma}$, define $pq$ as the concatenation of $p$ and $q$ i.e., $pq:= (e_1,\dots,e_n,f_1\dots, f_m)$ if $s(p)=t(q)$ (or equivalently  $s(e_n)=t(f_1)$) and $0$ otherwise, and extend this product bilinearly to any elements in $A[\Gamma]$. Under this product $A[\Gamma]$ becomes an $A$-algebra.
	\begin{remark}\label{rmk:altPA}
	Let $A^{\Gamma}$ be the free $A$-algebra on the set $V^{\Gamma} \cup E^{\Gamma}$ subject to the following relations
	\[ x^2=x, \ \ \ xy =0,\ \ \ x (x\xleftarrow{e} z)=x\xleftarrow{e} z=(x\xleftarrow{e} z)z,\]
	where $x, y, z \in V^{\Gamma}, x\neq y, e\in E^{\Gamma}$.
	The map $\iota: A^{\Gamma} \to A[\Gamma]$ given by $\iota(x)= x, \iota(e)=e$ is a well defined $A$-algebra isomorphism.
	\end{remark}
	\begin{proposition}\label{prop:AGammaidemp}
	Let $\Gamma = (V^\Gamma,E^\Gamma,s,t)$ be a quiver and $A$ a commutative unital ring, then $A[\Gamma]$ is an $A$-algebra with enough orthogonal idempotents given by $\{x\}_{x \in V^\Gamma}$ (the set of paths of length zero). Moreover if $|V^\Gamma|$ is finite then $A[\Gamma]$ is a unital $A$-algebra with identity $1_{A[\Gamma]}:=\sum_{x\in V^\Gamma} x$.
	\end{proposition}
	%%%%%%%%%%%%
		\subsection{Adjointness and bilinear forms}
			Let $A, B$ be rings. Let $M$ be an $(A,B)$-bimodule, $N$ an $B$-module, and $L$ an $A$-module. The following is a well-known result that will be used in Section \ref{sec:EandR}.
		\begin{proposition}\label{prop:adjoint}
			There is an isomorphism of (abelian) groups between $\Hom_A( M \otimes_B N,L)$ and $\Hom_B(N,\Hom_A(M,L))$  where $(\phi:M\otimes_B N\rightarrow L) \mapsto (\psi:N \rightarrow Hom_A(M,L))$ such that $\psi(n): M \rightarrow L$ maps $m$ to $\phi(m\otimes n)$ for all $n\in N, m\in M$. We say $\phi$ is the adjoint map of $\psi$.
		\end{proposition}
	%%%%%%%%%%%%
	%	\subsection{Bilinear forms}
			\begin{definition}Given two rings $A,B$, a left $A$-module ${}_AM$, a right $B$-module $M'_B$ and an $(A,B)$-bimodule ${}_AM''_B$, we define an $(A,B)$-bilinear form $<,>$ as an $(A,B)$-bimodule homomorphism from ${}_AM\times M'_B$ to ${}_AM''_B$. \end{definition}
	%%%%%%%%%%%%
	%%%%%%%%%%%%
	%%%%%%%%%%%%
\section{Hecke Datums and pre-$D$-graphs} \label{sec:IHAD}
In this section we extend the notion of $W$-graphs (introduced in \cite{KL1}) to \textit{pre-$D$-graphs} and \textit{$D$-graphs}, which are essentially $W$-graphs for multi-parameter Hecke algebras. We show these objects give rise to representations of Hecke algebras. Given a pre-$D$-graph (assuming mild finiteness conditions), we then construct a dual pre-$D$-graph  and show that their respective representations are contragredient (Proposition \ref{prop:dualrep}).
%%%%%%%%%
\subsection{Definition}\label{subsec:IHAD-definition}
%%%%%%%%%
Let $(W,S)$ be a Coxeter system, $A$ be a commutative unital ring, $(a_r)_{r\in S}$, $(b_r)_{r \in S}$ be families of elements in $A$ such that $a_r = a_s$, $b_r=b_s$ if $r$ and $s$ are conjugate in $W$. We say that a \textbf{Hecke datum} is a 5-tuple $D=(W,S,A, (a_r)_{r\in S}, (b_r)_{r \in S})$. Associated to each Hecke datum $D$ we have a unital Hecke algebra over the ring $A$, denoted $\Hecke(D)$, defined as
	\[  \left< \T_r \,\big|\, (\T_r\T_s\T_r \dots)_{m_{r,s}} - (\T_s\T_r\T_s\dots)_{m_{r,s}}=0, (\T_r-a_r)(\T_r+b_r)=0 \right>,\]
where $r, s \in S$, $m_{r,s}$ denotes the order of $rs$ in $W$, $(\T_r\T_s\T_r \dots)_{m_{r,s}}$ means that we have $m_{r,s}$ generators in the product, and we only consider such products when $m_{r,s}<\infty$. For any $w \in W$ with reduced expression $w=r_1r_2 \cdots r_n$ we set $\T_w := \T_{r_1}\T_{r_2}\cdots \T_{r_n}$. This expression does not depend on the reduced expression of $w$ since two reduced expressions differ by repeated application of braid relations of the form $(rsr \dots)_{m_{r, s}}=(srs \dots)_{m_{r, s}}$.

 For much of Section $\ref{universalPDG}$ we specialize to the datum $D_{\mathbb{Z}}:=(W,S,\mathbb{Z}[v,v^{-1}],\allowbreak v,v^{-1})$ (here $a_r=v$, $b_r=v^{-1}$ for all $r \in S$). In this case the algebra $\mathcal{H}(D_{\mathbb{Z}})$ is the equal parameter Hecke algebra associated to $(W,S)$ (see \cite{KL1}).
	\begin{remark}\label{rmk:alternateH}
	The Hecke algebra $\mathcal{H}(D)$ has an alternate description as the $A$-module generated by $\{\T_w\}_{w \in W}$ with multiplication defined as $\T_r\T_w = \T_{rw}$ if $l(w)<l(rw)$, and $\T_r\T_w = (a_r-b_r)\T_{w}+a_rb_r\T_{rw}$ if $l(w)>l(rw)$, where $r \in S, w \in W$, and $l$ is the length function of the Coxeter system $(W,S)$.
	\end{remark}
	\begin{definition}\label{def:preDgraph}
		Let $\mathcal{P}(S)$ be the power set of $S$. For any Hecke datum $D$, we define a \textbf{pre-$D$-graph} as a quiver $\Gamma$ with vertex set $V^{\Gamma}$, edge set $E^{\Gamma}$ together with a map $\mathcal{L}^{\Gamma}:V^{\Gamma} \rightarrow \mathcal{P}(S)$ such that 
		\begin{enumerate}
			\item for any $x \in V^{\Gamma}, r \in S$ with $r \not\in \mathcal{L}^{\Gamma}(x)$, there are only finitely many $y \in V^{\Gamma}$, $e \in E^{\Gamma}$ such that $r\in \mathcal{L}^{\Gamma}(y)$ and $x\xleftarrow{\;e\;} y $. \label{WDHecke1}
		\end{enumerate}
	A pre-$D$-graph is called \textbf{dualizable} if
		\begin{enumerate}
		\setcounter{enumi}{1}	
			\item for any $x \in V^{\Gamma}, r \in S$ with $r \in \mathcal{L}^{\Gamma}(x)$, there are only finitely many $y \in V^{\Gamma}$, $e \in E^{\Gamma}$ such that $r \not \in \mathcal{L}^{\Gamma}(y)$ and $y\xleftarrow{\;e\;} x$. \label{WDHecke2} 
		\end{enumerate}
	Conditions \ref{WDHecke1} and \ref{WDHecke2} hold automatically if both $V^{\Gamma}$ and $E^{\Gamma}$ are finite.\\
	A \textbf{$D$-graph} is a pre-$D$-graph $\Gamma$ together with a map $\mu^{\Gamma}:E^{\Gamma} \rightarrow A$ such that
		\begin{enumerate}
			\setcounter{enumi}{2}
			\item\label{WDHecke3}  for $\mathcal{E}^{\Gamma}$, a free unital $A$-module on the set $V^{\Gamma}$, and for all $ r, s \in S$ with $r \neq s, m_{r,s}<\infty$ we have 
			\[ (\tau_r^{\Gamma} \tau_s^{\Gamma} \tau_r^{\Gamma} \dots)_{m_{r,s}} - ( \tau_s^{\Gamma} \tau_r^{\Gamma} \tau_s^{\Gamma} \cdots)_{m_{r,s}} = 0 \in End_A(\mathcal{E}^\Gamma),\] 
		\end{enumerate}
where $\tau_r^{\Gamma}$ is the endomorphism of $\mathcal{E}^{\Gamma}$ given by
		\begin{equation}\label{eq:deftaur} \tau_r^{\Gamma} (x) = \begin{cases}
				(a_r) x + \sum\limits_{\substack{y ,e |\,r \in \mathcal{L}(y)\\ (x \xleftarrow{e} y)\in E^{\Gamma}}} \mu^{\Gamma}(e) y & \text{ if $r \not\in \mathcal{L}^{\Gamma}(x)$}\\
				(-b_r) x & \text{ if $r \in \mathcal{L}^{\Gamma}(x),$}\end{cases}
		\end{equation}
		for all $x \in V^{\Gamma}.$ We say a $D$-graph is dualizable if it is dualizable as a pre-$D$-graph.
	\end{definition}
	\begin{proposition}\label{prop:tauwd}
		Let $\Gamma$ be a $D$-graph then the map $\tau^{\Gamma}: \mathcal{H}(D)  \rightarrow End_A(\mathcal{E}^{\Gamma})$ where $\tau^{\Gamma}(\T_r)=  \tau^{\Gamma}_r$ is a representation of the Hecke algebra $\mathcal{H}(D)$. When no confusion arises, we will drop all superscripts $\Gamma$.
	\end{proposition}
	\begin{proof} 
		Property $\ref{WDHecke1}$ in Definition \ref{def:preDgraph} shows $\tau_r$ is a well defined endomorphism of $\mathcal{E}$ (the sum in $(\ref{eq:deftaur})$ is finite). To show the map $\tau$ is well defined, by property $\ref{WDHecke3}$, it suffices to show $(\tau_r-a_r \id_{\mathcal{E}})(\tau_r+b_r \id_{\mathcal{E}})=0$. First notice
			\[(\tau_r-a_r \id_{\mathcal{E}})(\tau_r+b_r \id_{\mathcal{E}})=\tau_r^2-a_r\tau_r+b_r\tau_r - a_rb_r \id_{\mathcal{E}}=(\tau_r+b_r \id_{\mathcal{E}})(\tau_r-a_r \id_{\mathcal{E}}).\] 
		Now if $x \in V$ such that $r \in \mathcal{L}(x)$, then 
			\[ (\tau_r-a_r \id_{\mathcal{E}})(\tau_r+b_r \id_{\mathcal{E}})(x) = (\tau_r-a_r \id_{\mathcal{E}})(0)=0.\]
If  $r \not\in \mathcal{L}(x)$, then 
		\[ (\tau_r+b_r \id_{\mathcal{E}})(\tau_r-a_r \id_{\mathcal{E}})(x) =(\tau_r+b_r \id_{\mathcal{E}})\bigg( \sum\limits_{\substack{y ,e |\,r \in \mathcal{L}(y)\\ (x \xleftarrow{e} y)\in E}} \mu(e) y\bigg)=0.\]
	\end{proof}
	\begin{remark}\label{rmk:gammared}
		For any $D$-graph $\Gamma$ we consider the $D$-graph $\Gamma^{red}$ with vertex set $V^{\Gamma^{red}}:=V^\Gamma$, edge set $E^{\Gamma^{red}}:= \{e\in E^\Gamma \ |  \ \mu^\Gamma(e)\neq0\}$ together with the restrictions of $\mathcal{L}^\Gamma$ and $\mu^\Gamma$ to $E^{\Gamma^{red}}\subset E^\Gamma$. It is easy to see that the representations $\tau^{\Gamma}$ and $\tau^{\Gamma^{red}}$ of $\mathcal{H}(D)$ are equal as $\tau_r^{\Gamma}=\tau_r^{\Gamma^{red}}$  for all $r\in S$.
	\end{remark}
	The following definition will be used in Proposition \ref{prop:dualrep} and Remark $\ref{rmk:dualiso}$.
	\begin{definition}\label{def:defiso}
		Given two pre-$D$-graphs $\Gamma$ and $\Lambda$, an isomorphism of pre-$D$-graphs is a tuple $(i,j)$, where $i:V^{\Gamma} \rightarrow V^{\Lambda}$ and $j:E^{\Gamma} \rightarrow E^{\Lambda}$ are bijections such that $i(s(e)) = s(j(e))$ and $i(t(e)) = t(j(e))$ for all $e \in E^{\Gamma}$, and $\mathcal{L}^{\Gamma}(x)=\mathcal{L}^{\Lambda}(i(x))$ for all $x \in V^{\Gamma}$. If $\Gamma$ and $\Lambda$ are $D$-graphs, then an isomorphism of $D$-graphs is an isomorphism $(i,j)$ of pre-$D$-graphs such that $\mu^{\Gamma}(e) = \mu^{\Lambda}(j(e))$. 
			\end{definition}
	\begin{remark}\label{rmk:KL}
		The notion of a $D$-graph can be regarded as a generalization of the $W$-graphs introduced by Kazhdan and Lusztig in \cite{KL1}. In what follows, we adopt the notation introduced by Lusztig in \cite{L2}. Let $(W,S)$ be a Coxeter system and $D_{\mathbb{Z}}=(W,S,\mathbb{Z}[v,v^{-1}], v,v^{-1})$ (here $a_r=v$ and $b_r = v^{-1}$ for all $r\in S$). Given a $W$-graph $\Lambda$ with vertex set $V^{\Lambda}$, edge set $E^{\Lambda}$, together with maps $\mu^{\Lambda}$ (with image in $\mathbb{Z}\setminus \{0\})$, and $\mathcal{L}^{\Lambda}$ as in \cite{KL1}, we can construct a $D_{\mathbb{Z}}$-graph $\Gamma$ where $V^{\Gamma} :=V^{\Lambda}, E^{\Gamma} := \{ x\xleftarrow{}y, y \xleftarrow{} x | \{x,y\} \in E^{\Lambda}\}$, and maps $\mu^{\Gamma}(x \xleftarrow{} y) := \mu^{\Lambda}(\{x,y\})$ and $\mathcal{L}^{\Gamma} := \mathcal{L}^{\Lambda}$. In that case, the representation $\tau^{\Gamma}$ defined above is the same as the one defined by Kazhdan and Lusztig, as both $\tau_r^{\Gamma}$ (Definition \ref{def:preDgraph}) and $\tau_r$ (see \cite{KL1}) are defined by the same formula (after a change of basis described by Lusztig in \cite{L2}).
	\end{remark}

\subsection{Duality}\label{subsec:IHAD-duality}

Let $D=(W,S,A, (a_r)_{r\in S}, (b_r)_{r \in S})$ be a Hecke datum and $\Gamma$ a  dualizable $D$-graph. We define $D^d$ to be the Hecke datum $(W,S,A,\allowbreak (-b_r)_{r\in S}, (-a_r)_{r \in S})$. Notice that $\mathcal{H}(D)$ is equal to $\mathcal{H}(D^d)$ since they are defined by the same presentation. Let $\Gamma^d$ be the quiver $(V^{\Gamma^d}, E^{\Gamma^d}, s,t)$ where 
		\[ V^{\Gamma^d}:= \{ x^d\ |\ x \in V^{\Gamma}\} \text{\ \ and \ \ } E^{\Gamma^d}:=\{ y^d \xleftarrow{e^d} x^d\ \big| \ x\xleftarrow {e} y \in E^{\Gamma}\}.\]
For any subset $S_1\subseteq S$, let $S_1^\complement$ be its complement in $S$. We define 
	\[\ldes^{\Gamma^d}(x^d \;) := \ldes^{\Gamma}(x)^{\complement} \subset \mathcal{P}(S) \text{\ and \ }\mu^{\Gamma^d}(y^d\xleftarrow{e^d} x^d) := \mu^{\Gamma}(x \xleftarrow{e}y).\]
Similar to Eq. (\ref{eq:deftaur}), we have the endomorphism $\tau^{\Gamma^d}_r$ of $\mathcal{E}^{\Gamma^d}$ given by
	\begin{equation}\label{eq:deftaudr}
		 \tau^{\Gamma^d}_r (x^d) = \begin{cases}
				(-b_r) x^d + \sum\limits_{\substack{y^d, e^d \,|\,r \in \mathcal{L}^{\Gamma^d}(y^d)\\ (x^d \xleftarrow{e^d} y^d)\in E^{\Gamma^d}}} \mu(e^d) y^d & \text{ if $r \not\in \mathcal{L}^{\Gamma^d}(x^d)$}\\
				(a_r) x^d & \text{ if $r \in \mathcal{L}^{\Gamma^d}(x^d).$}\end{cases}
	\end{equation}
 The fact that $\Gamma$ satisfies property $\ref{WDHecke2}$ in Definition \ref{def:preDgraph} guarantees that $\tau_r^{\Gamma^d}$ is an endomorphism of $\mathcal{E}^{\Gamma^d}$ (the sum in $(\ref{eq:deftaudr})$ is finite). 
 Let $<\cdot,\cdot>: \mathcal{E}^{\Gamma} \times \mathcal{E}^{\Gamma^d} \rightarrow A$ be the $(A,A)$-bilinear form given by $<x,y^d> = \delta_{x,y}$, where $\delta_{x,y}$ is the Kronecker delta function. Let $(\cdot)^\flat: h \mapsto h^\flat$ be unique involutive $A$-algebra anti-automorphism of $\mathcal{H}(D)$ that sends $\T_w$ to $\T_{w^{-1}}$ for any $ w \in W$ (see \cite{L1}, 3.4). The next proposition establishes a duality between $\Gamma$ and $\Gamma^d$.

	\begin{proposition}\label{prop:dualrep}
		Let $\Gamma$ be a dualizable $D$-graph, then 
		\begin{enumerate}[(I)]
			\item The quiver $\Gamma^d$, together with $\mathcal{L}^{\Gamma^d}$ and $\mu^{\Gamma^d}$, is a dualizable $D^d$-graph. We say that $\Gamma^d$ is the dual graph of $\Gamma$.\label{property:dualisDgraph}
			\item The representations of $\mathcal{H}(D) (= \mathcal{H}(D^d))$ given by $\tau^{\Gamma}$ and $\tau^{\Gamma^d}$ are contragredient in the following sense
			\small{\[ <\tau^{\Gamma}(h)(x),y^d> = <x,\tau^{\Gamma^d}(h^\flat)(y^d)> \text{ for all } h\in \mathcal{H}(D), x \in \mathcal{E}^{\Gamma}, y^d \in \mathcal{E}^{\Gamma^d}.\]}\label{property:dualcontra}
			\item $(\Gamma^d)^d$ is isomorphic to $\Gamma$ as $D$-graphs.\label{property:dualdual}
		\end{enumerate}
		It follows that if you replace $D$-graph (resp. $D^d$-graph) with pre-$D$-graph (resp. pre-$D^d$-graph) then the statements also hold.
	\end{proposition}
	\begin{proof}
		 For simplicity, for any $r\in S$, denote $\tau^{\Gamma}_r$ as $\tau_r$ and $\tau_r^{\Gamma^d}$ as $\tau_r^d$. We first show that for any $r_1, \dots, r_n \in S, x \in \mathcal{E}^{\Gamma}$, $y^d \in \mathcal{E}^{\Gamma^d}$ we have 
			\begin{equation}\label{eq:tauduality}
				 <\tau_{r_1}\cdots \tau_{r_n}(x),y^d> = <x,\tau^d_{r_n}\cdots \tau^d_{r_1}(y^d)>. 
			\end{equation}
		When $n=1$, this follows from equations (\ref{eq:deftaur}) and (\ref{eq:deftaudr}). 
		The general case then follows by induction. We now prove $(\ref{property:dualisDgraph})$ by showing the three properties in Definition \ref{def:preDgraph} hold for $\Gamma^d$. Property $\ref{WDHecke1}$ (resp.  $\ref{WDHecke2}$) for $\Gamma^d$ is equivalent to property $\ref{WDHecke2}$ (resp.  $\ref{WDHecke1}$) for $\Gamma$. To prove property $\ref{WDHecke3}$ notice that the left and right radicals of $<\cdot,\cdot>$ are trivial. Since property $\ref{WDHecke3}$ is satisfied for $\Gamma$, 
		\begin{align*}		
			0&=<(\tau_r \tau_s \tau_r \dots)_{m_{r,s}}(x) - ( \tau_s \tau_r \tau_s \cdots)_{m_{r,s}}(x), y^d>\\
			&=(-1)^{m_{r,s}-1}<x,(\tau^d_r \tau^d_s \tau^d_r \dots)_{m_{r,s}}(y^d) - ( \tau^d_s \tau^d_r \tau^d_s \cdots)_{m_{r,s}}(y^d)>,
		\end{align*}
		which implies property $\ref{WDHecke3}$ is satisfied for $\Gamma^d$. 
To prove $(\ref{property:dualcontra})$, by Proposition \ref{prop:tauwd},  $\tau^{\Gamma^d}$ is a representation of  $\mathcal{H}(D)$. The result now follows by Eq. (\ref{eq:tauduality}) and bilinearity of $<\cdot, \cdot>$. To prove $(\ref{property:dualdual})$ let $i: V^{\Gamma} \rightarrow V^{(\Gamma^d)^d}$ such that $i(x)=(x^d)^d$ for all $x\in V^{\Gamma}$, and $j:E^{\Gamma} \rightarrow E^{(\Gamma^d)^d}$ such that $j(x\xleftarrow{e} y)=(x^d)^d \xleftarrow{(e^d)^d} (y^d)^d$  for all $e\in E^{\Gamma}$. This pair satisfies the conditions in Definition \ref{def:defiso}. 
	\end{proof}

%%%%%
%%%%%%
%%%%%%%
 \section{Quotient Path Algebras}\label{sec:QPA}
 In this section we discuss the main topic of this paper, quotient path algebras. Throughout the section we let $D=(W,S,A, (a_r)_{r\in S}, (b_r)_{r \in S})$ be a Hecke datum (unless otherwise noted). We consider a pre-$D$-graph $\Gamma$ and construct its path algebra $A_0[\Gamma]$ (over a subring $A_0\subset A$ that satisfies certain conditions). We then quotient $A_0[\Gamma]$ by a specific ideal $J_0$ (Definition \ref{def:J0}) that allow us to construct a ``universal" representation of the Hecke algebra $\mathcal{H}(D)$ on $\mathcal{R}:=A\otimes_{A_0}A_0[\Gamma]/J_0$ (Proposition \ref{prop:rhodef}). We show $\mathcal{R}^{op}$ appears applying the same construction to the dual graph of $\Gamma$, and the representations of $\mathcal{H}(D)$ on $\mathcal{R}$ and $\mathcal{R}^{op}$ are contragredient (Theorem \ref{thm:Rdual}). We also show $\mathcal{R}$ admits several bimodule structures (Theorem \ref{thm:bimod}).
\subsection{Introduction}\label{sec:QPA:subsec:Intro}
In what follows, let $(D,A_0)$ be a tuple where $D=(W,S,A,\allowbreak (a_r)_{r\in S}, (b_r)_{r \in S})$ is a Hecke datum and $A_0$ is a unital subring of $A$ (with $1_{A_0}=1_A$) such that $A$ is a free $A_0$-module with some basis $\{c_i\}_{i\in I}$.  We say $(D,A_0)$ is an \textbf{extended Hecke datum}. The main example discussed in this paper (Section \ref{universalPDG}) is the extended Hecke datum $(D_\mathbb{Z},\mathbb{Z})$ introduced in Remark \ref{rmk:KL}, where we consider $\mathbb{Z}[v,v^{-1}]$ as a free $\mathbb{Z}$-module with basis $\{v^i\}_{i \in \mathbb{Z}}$.

 Let $\Gamma$ be a pre-$D$-graph with vertex set $V^{\Gamma}$, edge set $E^{\Gamma}$, and $\mathcal{L}^{\Gamma}: V^{\Gamma} \rightarrow A$. Let $P^{\Gamma}$ be the set of paths of $\Gamma$ and $A_0[\Gamma]$ be the path algebra of $\Gamma$ over $A_0$. Extending scalars we can identify $ A\otimes_{A_0}A_0[\Gamma]$ with $A[\Gamma]$, the path algebra of $\Gamma$ over $A$. For simplicity we denote the element $a\otimes p$ as $ap$. We would like to define a representation of the Hecke algebra $\mathcal{H}(D)$ given by 
%	\begin{align*}
			$\rho^{\Gamma}: \mathcal{H}(D)  \rightarrow End_{A[\Gamma]}(A[\Gamma])$ such that $\rho^{\Gamma}(\T_r)=  \rho_r^{\Gamma}$
%	\end{align*} 	
where
\begin{equation}\label{eq:rhodeforig} \rho_r^{\Gamma} (x) = \begin{cases}
				(a_r) x + \sum\limits_{\substack{y,e\,|\,x \xleftarrow{e} y \\r \in \mathcal{L}^{\Gamma}(y)}} (x \xleftarrow{e}y) & \text{ if $r \not\in \mathcal{L}^{\Gamma}(x)$}\\
				(-b_r) x & \text{ if $r \in \mathcal{L}(x),$}\end{cases}\end{equation}
for $x \in V^\Gamma$. Again, when no confusion arises, we will drop the superscripts $\Gamma$.

By Proposition \ref{prop:quotientidemp}$(\ref{property:isoend})$ we have $\rho_r \in End_{A[\Gamma]} A[\Gamma]\subseteq End_{A} A[\Gamma]$.  In general, the map $\rho$ is not well defined since the braid relations are not satisfied. Below, the general idea is to quotient $A_0[\Gamma]$ by the smallest (two-sided) ideal $J_0^{\Gamma}$ such that $\rho$ induces a representation of $\mathcal{H}(D)$ on the $A$-algebra $A\otimes_{A_0}(A_0[\Gamma]/J_0^{\Gamma})$. For $x\in V$ and $r,s \in S$ with $r\neq s$, $m_{r,s}<\infty$, let $ X_{p,i}^{r,s,x} \in A_0$ be defined such that
	\begin{equation}\label{eq:relations}
		\left((\rho_r \rho_s \rho_r \dots)_{m_{r,s}}-(\rho_s \rho_r \rho_s \dots)_{m_{r,s}}\right) (x) = \sum\limits_{i\in I}\sum\limits_{p\in P}c_i X_{p,i}^{r,s,x}  p.
	\end{equation}
	To simplify notation let $S^2_{fin} := \{ (r,s) \in S\times S \,|\, m_{r,s}<\infty\}$ (introduced in \cite{BB}).
	\begin{definition}\label{def:J0}
		Let $J_0^\Gamma$ be the two-sided ideal of $A_0[\Gamma]$ generated by the set $\{\sum_{p} X_{p,i}^{r,s,x}p \,|\, x \in V^{\Gamma}, i \in I, (r,s)\in S^2_{fin}\}$.
	\end{definition}
A priori the ideal $J_0^\Gamma$ depends on the basis $\{c_i\}_{i \in I}$, however if $\{c_i\}_{i \in I}$ and $\{d_j\}_{j \in J}$ are bases of $A$ as a free $A_0$-module, then there exist $k_{i_j} \in A_0$ such that $d_j=\sum_i k_{i_j}c_i$ for all $j$. Write
		\[\left((\rho_r \rho_s \rho_r \dots)_{m_{r,s}}-(\rho_s \rho_r \rho_s \dots)_{m_{r,s}}\right) (x) = \sum\limits_{j\in J}\sum\limits_{p\in P}d_j Y_{p,j}^{r,s,x}  p,\]
 where $Y_{p,j}^{r,s,x} \in A_0$, then
			\[ \sum\limits_{i,p}c_i X_{p,i}^{r,s,x} p= \sum\limits_{j,p} d_j Y_{p,j}^{r,s,x}  p=\sum_i \bigg(c_i \sum_jk_{i_j}\bigg(\sum_p Y_{p,j}^{r,s,x}p\bigg) \bigg).\]
		Since $\{c_i\}_{i \in I}$ is a basis, we obtain $\sum_{p}X_{p,i}^{r,s,x}p =  \sum_j k_{i_j}\left(\sum_p Y_{p,j}^{r,s,x}p\right)$ for any (fixed) $i \in I$, hence the ideal generated using the basis $\{c_i\}_{i\in I}$ is contained in the ideal generated using the basis $\{d_j\}_{j\in J}$. The reverse inclusion follows by the same argument interchanging the bases. Thus $J_0^\Gamma$ is independent of the basis of $A$ as a free $A_0$-module.

	\begin{definition}\label{def:qpa}
		Let  $J^{\Gamma}:= A \otimes_{A_0} J_0^{\Gamma}, \mathcal{R}_0^{\Gamma}:=A_0[\Gamma]/J_0^{\Gamma}, \mathcal{R}^{\Gamma}:=A\otimes_{A_0} \mathcal{R}_0^{\Gamma}$. We say $\mathcal{R}^{\Gamma}$ is the \textbf{quotient path algebra} associated to $\Gamma$.
	\end{definition}
	Again, when no confusion arises, we will drop the superscripts $\Gamma$. Since $A$ is free over $A_0$ (hence flat over $A_0$), then $\mathcal{R}$ can be identified with $(A\otimes_{A_0} A_0[\Gamma])/(A\otimes_{A_0}J_0)$ and furthermore with $A[\Gamma]/J$. To distinguish a class $p + J \in \mathcal{R}$ from the element $p \in A[\Gamma]$, we will use the notation $[p] \in \mathcal{R}$. In general $\mathcal{R}$ may be trivial, however in Sections \ref{sec:QPA:subsec:Equivariantmap} and \ref{subsec:examples} we discuss pre-$D$-graphs which give rise to non-trivial quotient path algebras.  
		
	By Proposition \ref{prop:quotientidemp}$(\ref{property:quotientidemp})$ the $A$-algebra $\mathcal{R}$ has enough orthogonal idempotents given by $\{ [x]\}_{x \in V}$, and by Proposition \ref{prop:quotientidemp}$(\ref{property:isoend})$, in order to define a left $\mathcal{R}$-module homomorphism of $\mathcal{R}$, we only need to specify the value of the map on the elements $[x]$ for all $x \in V$. For any $x\xleftarrow{e}y \in E$ define $\eta(e) := [x \xleftarrow{e} y] \in \mathcal{R}$, and let $\trho_r^{\Gamma}$ be the left $\mathcal{R}$-module endomorphism of $\mathcal{R}$ such that
\[ \trho_r^{\Gamma} ([x]) = \begin{cases}
				(a_r) [x] + \sum\limits_{\substack{y,e|x \xleftarrow{e} y \\r \in \mathcal{L}(y)}}\eta(e)[y] & \text{ if $r \not\in \mathcal{L}(x)$}\\
				(-b_r) [x] & \text{ if $r \in \mathcal{L}(x),$}\end{cases}\]
for all $x \in V$. Even though $\eta(e)[y]=[x \xleftarrow{e} y]$, we write the definition of $\trho_r$ in that form to resemble the definition of $\tau_r$ in $(\ref{eq:deftaur})$. Let $\pi:A[\Gamma]\rightarrow \mathcal{R}$ be the natural quotient map i.e., $\pi(p) = [p] \in \mathcal{R}$. By definition of $\rho_r$ and $\trho_r$, we get $\pi(\rho_r(x)) = \trho_r(\pi(x))$ for any $x \in V$.
	\begin{proposition}\label{prop:rhodef}
There is a representation $\trho: \mathcal{H}(D) \rightarrow End_A(\mathcal{R})$ of $\mathcal{H}(D)$ such that $\trho(\T_r)=  \trho_r^{\Gamma}$. When no confusion arises, we will write $\trho_r$ for $\trho_r^{\Gamma}$.
	\end{proposition}
	\begin{proof}
Since $End_{\mathcal{R}}\mathcal{R}\subseteq End_A\mathcal{R}$ we have $\trho_r \in End_A\mathcal{R}$. It now suffices to show that the relations defining $\mathcal{H}(D)$ are preserved. Let $(r,s) \in S^2_{fin}$ and $m:=m_{r,s}$. Then for $[p]\in \mathcal{R}$ we have 
	\small{\[\left((\trho_r \trho_s \trho_r \dots)_m-(\trho_s \trho_r \trho_s \dots)_m\right) ([p]) = \pi((\rho_r \rho_s \rho_r \dots)_m-(\rho_s \rho_r \rho_s \dots)_m (p)) \in J, \]}
thus $(\trho_r \trho_s \trho_r \dots)_m-(\trho_s \trho_r \trho_s \dots)_m =0$. The relation $(\trho_r- a_r \id_{\mathcal{R}})(\trho_r + b_r \id_{\mathcal{R}})=0$ follows from an entirely similar computation to the one in Proposition \ref{prop:tauwd}. 
	\end{proof}
	 In Secion \ref{sec:QPA:subsec:Equivariantmap} we assert the connection between the representation of $\mathcal{H}(D)$ given by $\tau$ (Proposition  \ref{prop:tauwd}) and the one given by $\trho$. However, this connection comes a special case of a more general theory described in Section \ref{sec:WGONCA}. Theorem \ref{thm:universality} supports the claim that $\trho$ is a universal representation, as any other representation of $\mathcal{H}(D)$ coming from a $D$-graph (or more generally a ``$D$-graphs over a non-commutative algebra", see Section \ref{sec:WGONCA}) can be realized as a ``specialization" of $\trho$. 
	\begin{remark}\label{rmk:dualiso}
		Let $(D,A_0)$ be an extended Hecke datum. If there is a pre-$D$-graph isomorphism $(i,j)$ from $\Gamma$ to $\Lambda$ then $\mathcal{R}^{\Gamma}_0 \cong \mathcal{R}^{\Lambda}_0$ with the isomorphism given by $\iota:[(x_0 \xleftarrow{e_1} x_1 \cdots \xleftarrow{e_n}x_n)]\mapsto [i(x_0) \xleftarrow{j(e_1)} i(x_2) \cdots \xleftarrow{j(e_n)}i(x_n)],$ and similarly $\mathcal{R}^{\Gamma} \cong \mathcal{R}^{\Lambda}$. This will be used in Theorem \ref{thm:Rdual}.
	\end{remark}

%%%%%%%%%%%%%%%%%%%%%%%%%%%%%%%%%%%%
\subsection{Explicit computation of $\trho_{r_n} \cdots \trho_{r_1}$}\label{subsec:comput}
Let $\Gamma$ be any pre-$D$-graph where $(D,\allowbreak A_0)$ is an extended Hecke datum. In some circumstances it is helpful to have a concrete description of the following composition, $\trho_{r_n} \cdots \trho_{r_1}([p])$ for all $p \in P^{\Gamma}.$ Since $\trho_r$ is a left $\mathcal{R}$-module homomorphism, it is enough to consider $\trho_r([x_0])$, for $x_0 \in V^{\Gamma}$. In order to give a more precise and elegant description, we first extend the quiver $\Gamma$ by adding distinguished loops in every vertex, that is, let $\Gamma^{ext}$ be the quiver with vertex set $V^{\Gamma^{ext}}:=V^{\Gamma}$ and edge set $E^{\Gamma^{ext}}:=E \dot\cup \{ x \xleftarrow{e_x} x | x \in V\}$. Throughout this section we assume $q \in P^{\Gamma^{ext}}$. Define $\nu (q) \in P^{\Gamma}$ as the path that you get by removing the distinguished loops in $q$. For example, if $x,y \in V^{\Gamma}$, $x \xleftarrow{e}y, x \xleftarrow{f}x \in E^{\Gamma}$ then $\nu(x\xleftarrow{f} x \xleftarrow{e} y \xleftarrow{e_y} y) = x\xleftarrow{f} x \xleftarrow{e} y \in P^{\Gamma}$. 
	\begin{proposition}\label{prop:explicit}
		Fix $x_0$ in $V^{\Gamma}$, then for each $r_1, \dots, r_n \in S$ we have
		\begin{equation}\label{eq:explicitcompu}
			 \rho_{r_n} \cdots \rho_{r_1} (x_0) = \sum_{q \in \mathfrak{P}_n} \Xi(q) \nu(q) \in A[\Gamma],
		\end{equation}
where 
		\begin{enumerate}
			\item $\mathfrak{P}_n$ is the set of all paths $q =(x_0 \xleftarrow{e_1} x_1 \xleftarrow{e_2} \cdots  \xleftarrow{e_{n-1}}x_{n-1}\xleftarrow{e_{n}} x_n) \in P^{\Gamma^{ext}}$ such that for all $i\in \{0, \dots, n-1\}$ if $x_i\neq x_{i+1}$ then $r_{i+1} \not \in \mathcal{L}(x_i)$ and $r_{i+1} \in \mathcal{L}(x_{i+1})$, while if $x_i = x_{i+1}$ then $e_{i+1}=e_{x_i}.$\label{expcompcondition}
			\item $\Xi(q)$ is the product of the elements in the sequence $\{\xi_i(q)\}_{i=1}^{n}$ where 
		\[ \xi_i(q)=\begin{cases} a_{r_i}&\text{ if } x_{i-1}=x_i \text{ and } r_i \not \in \mathcal{L}(x_{i-1})\\-b_{r_i}&\text{ if } x_{i-1}=x_i \text{ and } r_{i} \in \mathcal{L}(x_{i-1})\\1_A&\text{ if } x_{i-1}\neq x_i. \end{cases}\]
\end{enumerate}
Moreover, 
\[\trho_{r_n} \cdots \trho_{r_1} ([x_0]) = \sum_{q \in \mathfrak{P}_n} \Xi(q) [\nu(q)] \in \mathcal{R}.\]
	\end{proposition}
	\begin{proof}
	Equation $(\ref{eq:explicitcompu})$ follows by induction. The last statement follows by applying the natural quotient map $\pi$.
\end{proof}

%%%%%%%%%%%%%%%%%%%%%%%%%%%%%%%%%%%%
\subsection{Bimodule structures on $\mathcal{R}$}\label{seq:QPA:subsec:bimod}
In this section, we describe several bimodule structures on certain specializations of the Hecke algebra. We then discuss some of the most interesting examples. To be specific, 
let $(D,A_0)$ be an extended Hecke datum, $\Gamma$ a pre-$D$-graph with finitely many vertices and edges, and $\mathcal{R}$ the quotient path algebra associated to $\Gamma$. Let $1_{\mathcal{R}_0} := \sum_{x \in V} [x]$ be the identity in $\mathcal{R}_0$, thus $1_{\mathcal{R}}:=1_A1_{\mathcal{R}_0}$ is the identity in $\mathcal{R}$. As defined in Proposition \ref{prop:rhodef}, the map $\trho: \mathcal{H}(D) \rightarrow End_{\mathcal{R}}(\mathcal{R})$ defines a representation of the $A$-algebra $\mathcal{H}(D)$. Since $End_{\mathcal{R}} (\mathcal{R}) \cong \mathcal{R}^{op}$, the map $\trho$ can be regarded as an $A$-algebra anti-homomorphism from $\mathcal{H}(D)$ to $\mathcal{R}$, where $\T_r \mapsto \trho_r(1_{\mathcal{R}})$. 

Let $B$ be a commutative unital $A_0$-algebra and $f:A  \to B$ be an $A_0$-algebra homomorphism. We can now construct a specialization of $\mathcal{H}(D)$ on $B$ using the map $f$. The $A_0$-algebra $B$ has a right $A$-module structure given by $b\cdot a=bf(a)$ for all $a\in A, b\in B$. Let $\mathcal{H}(D)':= B\otimes_A \mathcal{H}(D)$ and $\mathcal{R}':= B\otimes_A\mathcal{R}$. Since $\mathcal{R}= A\otimes_{A_0} \mathcal{R}_0$ then $\mathcal{R}'\cong B\otimes_{A_0} \mathcal{R}_0$ as $B$-algebras. Extending the $A$-algebra anti-homomorphism $\trho: \mathcal{H}(D) \to \mathcal{R}$ we get a $B$-algebra anti-homomorphism $\trho':\mathcal{H}(D)' \to \mathcal{R}'$ that sends $b\T_r$ to $b\trho'_r(1_{\mathcal{R}_0})$, where for any $x \in V^{\Gamma}$ we have 
		\begin{equation}\label{eq:trhoprime}
			\trho_r' ([x]) = \begin{cases}
				f(a_r) [x] + \sum\limits_{\substack{y,e|x \xleftarrow{e} y \\r \in \mathcal{L}(y)}}[x\xleftarrow{e}y] & \text{ if $r \not\in \mathcal{L}(x)$}\\
				-f(b_r) [x] & \text{ if $r \in \mathcal{L}(x).$}\end{cases}
		\end{equation}
The following theorem provides three bimodule structures on the $B$-algebra $\mathcal{R}'$. A special case of this result should be compared to 14.15 in \cite{L1}, as it is pointed out in Remark \ref{rmk:vv'}.
	\begin{theorem}\label{thm:bimod}
		Let $f_1,f_2:A\to B$ be two $A_0$-algebra homomorphisms, and  $\mathcal{H}(D)'$ (resp. $\mathcal{H}(D)''$) be the specialization of $\mathcal{H}(D)$ on $B$ using $f_1$ (resp. $f_2$). Then the $B$-algebra $\mathcal{R}'=B \otimes_{A_0} \mathcal{R}_0$ has an $(\mathcal{H}(D)', \mathcal{H}(D)'')$-bimodule structure given by  
			\[T_r' \cdot b[p] \cdot T_r'' = b\bigg(\trho''_r(1_{\mathcal{R}_0})\bigg)[p]\bigg(\trho_r'(1_{\mathcal{R}_0})\bigg),\]
		where $\trho_r', \trho_r''$ are given as in Eq. $(\ref{eq:trhoprime})$ (replacing $f$ with $f_1$ and $f_2$ respectively). Moreover, $\mathcal{R}'$ has an $(\mathcal{H}(D)',\mathcal{R}')$-bimodule structure and an $(\mathcal{R}',\mathcal{H}(D)'')$, where both (left and right) actions of $\mathcal{R}'$ on itself are given by multiplication.
	\end{theorem}
	 \begin{proof}
	 Let $B'$ (resp. $B''$) be the $A$-module $B$ via the map $f_1$ (resp. $f_2$). Notice that both $B'\otimes_A (A\otimes_{A_0}\mathcal{R}_0)$ and $B''\otimes_A (A\otimes_{A_0}\mathcal{R}_0)$ are isomorphic to $\mathcal{R}'=B\otimes_{A_0}\mathcal{R}_0$ as $B$-algebras. Thus extending the map $\trho: \mathcal{H}(D) \to \mathcal{R}$ as described above we get $B$-algebra anti-homomorphisms $\mathcal{H}(D)'\to \mathcal{R}'$ (resp. $\mathcal{H}(D)''\to \mathcal{R}')$ sending $bT_r'$ to $b\trho'_r(1_{\mathcal{R}_0})$ (resp. $bT_r''$ to $b\trho''_r(1_{\mathcal{R}_0})$), where both $\trho_r'$ and $\trho_r''$ are given as in Eq. $(\ref{eq:trhoprime})$, replacing $f$ with $f_1$ and $f_2$ respectively. Thus we get the desired left $\mathcal{H}(D)'$ action and right $\mathcal{H}(D)''$ action on $\mathcal{R}'$. The fact that all actions commute (including those in the last statement) follows from associativity of multiplication in $\mathcal{R}'$. 
	 	 \end{proof}
One can check that the $A$-algebra $\mathcal{H}(D)'$ is isomorphic to the Hecke algebra $\mathcal{H}(D')$, where $D'$ is the datum $(W,S,B,(f_1(a_r))_{r\in S},(f_1(b_r))_{r\in S})$. Similarly, $\mathcal{H}(D)''$ is isomorphic to $\mathcal{H}(D'')$, where $D''=(W,S,B,(f_2(a_r))_{r\in S},\allowbreak(f_2(b_r))_{r\in S})$.
	\begin{corollary}
		Let $\Gamma$ be a pre-$D$-graph with finitely many vertices and edges, and $f$ be an $A_0$-algebra endomorphism of $A$, then we get an $(\mathcal{H}(D),\allowbreak\mathcal{H}(D)')$-bimodule structure on the quotient path algebra $\mathcal{R}$, where $\mathcal{H}(D)$ is the Hecke algebra associated to $D$ and $\mathcal{H}(D)'$ is its specialization via $f$. Moreover, $\mathcal{H}(D)'$ is isomorphic as $A$-algebra to the Hecke algebra $\mathcal{H}(D')$ where $D'=(W,S,A,(f(a_r))_{r\in S},(f(b_r))_{r\in S})$.
	\end{corollary}
	\begin{proof}
	Let $B=A, f_1=\id_A$, and $f_2=f$ in Theorem \ref{thm:bimod}.
	\end{proof}
	\begin{remark}\label{rmk:vv'}
	Let $A_0=\mathbb{Z}, A=\mathbb{Z}[v,v^{-1}]$, where $A$ is consider as a free $A_0$-module with basis $\{v^i\}_{i\in \mathbb{Z}}$. Let $a_r=v, b_r=v^{-1}$ for all $r\in S$. Let $v'$ be another indeterminate and consider the $A_0$-algebra $B=\mathbb{Z}[v,v^{-1},v',v'^{-1}]$ with $A_0$-algebra homomorphism $f_1:A\to B$ (resp. $f_2:A\to B$) taking $v^k$ to $v^k$ (resp. $v^k$ to $v'^k$) for every $k \in \mathbb{Z}$. The result of Theorem $\ref{thm:bimod}$ under these conditions should be compared to 14.15 in \cite{L1}. We can further specialize to $v'=1$ to get an $(\mathcal{H}(D),A[W])$-bimodule structure on $\mathcal{R}$.
	\end{remark}

%%%%%%%%%%%%%%%%%%%%%%%%%%%%%%%%%%%%
%
\subsection{Duality}\label{subsec:duality}

In this section we describe several connections between a dualizable pre-$D$-graph $\Gamma$ and its dual pre-$D$-graph $\Gamma^d$, as well as connections between the quotient path algebras associated to these two pre-$D$-graphs.
Let $D=(W,S,A,(a_r),(b_r))$. Let $(D,A_0)$ be an extended Hecke datum, and $\Gamma$ be a dualizable pre-$D$-graph with $\Gamma^d$ as the dual pre-$D^d$-graph (recall from Section \ref{subsec:IHAD-duality} that $D^d=(W,S,A,(-b_r),(-a_r))$). For $p =(x_0 \xleftarrow{e_1} \cdots \xleftarrow{e_n}x_n) \in P^{\Gamma}$, let $p^{d} :=(x_n^d \xleftarrow{e_n^d} \cdots \xleftarrow{e_1^d}x_0) \in P^{\Gamma^{d}}$. Let $\phi_0: A_0[\Gamma] \rightarrow A_0[\Gamma^d]$ be the $A_0$-algebra anti-homomorphism such that $\phi_0( p)=p^{d}$, and $\phi:=\id_A\otimes\phi_0: A[\Gamma]\rightarrow A[\Gamma^d]$. By the way we defined $\Gamma^d$, it is clear $\phi_0$ is an anti-isomorphism, thus so is $\phi$. From now on, for any $q\in A[\Gamma]$ let $q^d := \phi(q) \in A[\Gamma^d]$. The map $\phi$ defines a right $A[\Gamma]$-module structure on $A[\Gamma^d]$ given by $q^d\cdot p:=  p^dq^d$ for any $p \in A[\Gamma], q^d \in A[\Gamma^d]$. 

Consider the left $A[\Gamma]$-module structure on $A[\Gamma]$ given by  multiplication, then we get an $(A[\Gamma],A[\Gamma])$-bilinear form $<,>: A[\Gamma] \times A[\Gamma^d] \rightarrow A[\Gamma]$ such that $<p,q^d>:= p\phi^{-1}(q^d)=pq$. The following lemma will be key in proving certain duality properties in Theorem \ref{thm:Rdual}.
	\begin{lemma}\label{lemma:phionrho}
		For any $p \in A[\Gamma], q^d \in A[\Gamma^d]$ we have
			\[ <\rho^{\Gamma}_{r_n} \cdots \rho_{r_1}^{\Gamma} (p),q^d>=<p,\rho_{r_1}^{\Gamma^d} \cdots \rho_{r_n}^{\Gamma^d} (q^d)>.\]
	\end{lemma}
	\begin{proof}
		To prove the statement we first show it is true when $n=1$. Let $RHS$ (resp. $LHS$) be the right (resp. left) hand side of the equation. Let $p,q \in P^{\Gamma}$ (so $q^d\in P^{\Gamma^d}$), then it can be checked
		\small{\[ LHS=RHS=\begin{cases}
		 \sum_e p\big(s(p)\xleftarrow{e} t(q)\big)q & \tif r_1\not\in \mathcal{L}^{\Gamma}(s(p))  \tand r_1\not\in \mathcal{L}^{\Gamma^d}(s(q^d))\\ 
		-b_{r_1} pq & \tif r_1\in \mathcal{L}^{\Gamma}(s(p)) \tand r_1\not\in \mathcal{L}^{\Gamma^d}(s(q^d))\\
		 a_{r_1}pq & \tif r_1\not\in \mathcal{L}^{\Gamma}(s(p)) \tand r_1\in \mathcal{L}^{\Gamma^d}(s(q^d))\\
		 0 & \tif r_1\in \mathcal{L}^{\Gamma}(s(p)) \tand r_1\in \mathcal{L}^{\Gamma^d}(s(q^d)).
		 \end{cases}\]}
If $p,q$ are elements in $A[\Gamma]$ we use bilinearity of $<,>$ and the result above to get $<\rho_{r_1}^{\Gamma} (p),q^d>=<p,\rho_{r_1}^{\Gamma^d}(q^d)>$. An easy induction completes the proof.
	\end{proof}
	The following theorem provides several duality properties regarding a dualizable pre-$D$-graph $\Gamma$ and its dual $\Gamma^d$.
	\begin{theorem}\label{thm:Rdual}
		For any dualizable pre-$D$-graph $\Gamma$, the following statements hold.
			\begin{enumerate}[(I)]
				\item \label{property:isoR} $\mathcal{R}^{\Gamma^d}_0 \cong (\mathcal{R}^{\Gamma}_0)^{op}$ and $\mathcal{R}^{\Gamma^d} \cong (\mathcal{R}^{\Gamma})^{op}$ as $A_0$-algebras and $A$-algebras respectively. 
				\item \label{property:bilform}The $(A[\Gamma],A[\Gamma])$-bilinear form $<,>$ descends to a $(\mathcal{R}^{\Gamma},\mathcal{R}^\Gamma)$-bilinear form $\allowbreak <,>: \mathcal{R}^{\Gamma} \times \mathcal{R}^{\Gamma^d} \rightarrow \mathcal{R}^{\Gamma}$ such that $<[p],[q^d]>=[pq]$.
				
				\item \label{property:dualityR}The representations $\trho^\Gamma$ and $\trho^{\Gamma^d}$ of $\mathcal{H}(D)$ are contragredient in the following sense, for $[p]\in  \mathcal{R}^{\Gamma}, [q^{d}] \in \mathcal{R}^{\Gamma^d}, h \in \mathcal{H}(D)$ we have 
				\[<\trho^{\Gamma}(h)([p]),[q^{d}]>=<[p],\trho^{\Gamma^d}(h^\flat)([q^{d}])>,\] 
				where $(\cdot)^\flat$ is the $A$-algebra anti-isomorphism of $\mathcal{H}(D)$ used in Proposition \ref{prop:dualrep}. 
			\end{enumerate} 
	\end{theorem}
	\begin{proof}
		To prove $(\ref{property:isoR})$ first we show $\phi_0(J_0^{\Gamma}) \subseteq J_0^{\Gamma^d}$. 
		Let $\sum_{p} X_{p,i}^{r,s,x} p$ (resp. $\sum_{q} Y_{q,i}^{r,s,x^d} q$) be the generators of $J_0^{\Gamma}$ (resp. $J_0^{\Gamma^d}$) as in Eq. $(\ref{eq:relations})$. Since $\sum_{p} X_{p,i}^{r,s,x} p=\sum_{z\in V}(\sum_{p} X_{p,i}^{r,s,x} p)z$, then $J_0^{\Gamma}$ is generated by the set of elements of the form $\sum_{p} X_{p,i}^{r,s,x} pz$ where $z\in V^{\Gamma}$. Similarly, $J_0^{\Gamma^d}$ is generated by elements of the form $\sum_{q} Y_{q,i}^{r,s,x^d} qz^d$ where $z^d \in V^{\Gamma^d}$. By definition of $<,>$ we have 
		\begin{equation}\label{eq:bilform}
		\sum_{i,p} c_iX_{p,i}^{r,s,x} pz=<(\rho^{\Gamma}_r\rho^{\Gamma}_s\rho^{\Gamma}_r\dots)_{m_{r,s}} (x) - (\rho_s^{\Gamma}\rho_r^{\Gamma}\rho_s^{\Gamma}\dots)_{m_{r,s}} (x),z^d>.
		\end{equation}
		 By Lemma \ref{lemma:phionrho} this is also equal to 
		 	\[(-1)^{{m_{r,s}}-1}<x, (\rho^{\Gamma^d}_r\rho^{\Gamma^d}_s\rho^{\Gamma^d}_r\dots)_{m_{r,s}} (z^d)-  (\rho_s^{\Gamma^d}\rho_r^{\Gamma^d}\rho_s^{\Gamma^d}\dots)_{m_{r,s}} (z^d)>.\] 	
		By definition of $<,>$ this is equal to $(-1)^{m_{r,s}-1}\phi^{-1}(\sum_{i,q}c_iY_{q,i}^{r,s,z^d}qx^d)$. Picking out the coefficient of $c_i$ we get 
	\[\sum_{p} X_{p,i}^{r,s,x} pz=(-1)^{m_{r,s}-1}\phi_0^{-1}(\sum_{q}Y_{q,i}^{r,s,z^d}qx^d) \in J_0^{\Gamma},\] thus $\phi_0(J_0^{\Gamma}) \subseteq J_0^{\Gamma^d}$. Therefore, the map $\tphi_0: \mathcal{R}_0^{\Gamma} \rightarrow \mathcal{R}_0^{\Gamma^d}$ where $\tphi_0([p]) =[\phi_0(p)]=[p^{d}]$ is a well-defined $A_0$-algebra anti-homomorphism. Applying this result to $\Gamma^d$, one obtains an $A_0$-algebra anti-homomorphism $\tphi'_0:\mathcal{R}_0^{\Gamma^d} \rightarrow \mathcal{R}_0^{(\Gamma^d)^d}$. By Remark \ref{rmk:dualiso} we have an isomorphism $\iota: \mathcal{R}_0^{(\Gamma^d)^d} \rightarrow \mathcal{R}_0^{\Gamma}$. It is now straightforward to check that $\tphi_0$ and $\iota\circ \tphi_0'$ are inverses of each other, thus $\tphi_0$ is an $A_0$-algebra anti-isomorphism. It now follows that $\tphi:=\id_A\otimes \tphi_0: \mathcal{R}^{\Gamma} \rightarrow\mathcal{R}^{\Gamma^d}$ is an $A$-algebra anti-isomorphism.
				
		To prove $(\ref{property:bilform})$ let $(,): A[\Gamma]\times A[\Gamma^d]\rightarrow \mathcal{R}^{\Gamma}$ be the composition $\pi^{\Gamma}\circ <,>$ (so $(p,q^d)=[pq]$). It follows $(,)$ is an $(A[\Gamma],A[\Gamma])$-bilinear form. By Eq. $(\ref{eq:bilform})$ we have that $J^{\Gamma}$ is contained in the left radical of $(,)$. This, together with  Lemma \ref{lemma:phionrho},  implies that $J^{\Gamma^d}$ is contained in the right radical of $(,)$. Considering the left $\mathcal{R}^{\Gamma}$-module structure on $\mathcal{R}^{\Gamma}$ given by multiplication, and the right $\mathcal{R}^{\Gamma}$-module structure on $\mathcal{R}^{\Gamma^d}$ given by $[q^d]\cdot [p] := [p^dq^d]$ we get an $(\mathcal{R}^{\Gamma},\mathcal{R}^{\Gamma})$-bilinear form $<,>: \mathcal{R}^{\Gamma} \times \mathcal{R}^{\Gamma^d} \rightarrow \mathcal{R}^{\Gamma}$ such that $<[p],[q^d]>=(p,q^d)=[pq]$.
				 
		To prove $(\ref{property:dualityR})$, by Remark \ref{rmk:alternateH} $\mathcal{H}(D)$ is generated as an $A$-module by the elements $T_{w}:= T_{r_1} \cdots T_{r_n}$ (where $w=r_1\cdots r_n\in W$ is any reduced expression for $w$). By Lemma \ref{lemma:phionrho} and Theorem \ref{thm:Rdual}$(\ref{property:bilform})$ we have $<\trho^{\Gamma}(T_w)([p]),[q^{d}]>=<[p],\trho^{\Gamma^d}(T_{w^{-1}})([q^{d}])>$. Since $\trho^{\Gamma}$ and $\trho^{\Gamma^d}$ are $A$-algebra homomorphisms then we get the result for any $h \in \mathcal{H}(D)$.
\end{proof}
%%%%%%%%%%%%%%
%%%%%%%%%%%%%%
%%%%%%%%%%%%%%
\section{Relationship between various representations of $\mathcal{H}(D)$}\label{sec:EandR}
In this section we study the relationship between several representations of the Hecke algebra $\mathcal{H}(D)$ (given a $D$-graph $\Gamma$). We first show there is an equivariant map between the representations $\tau$ and $\trho$ (Theorem \ref{thm:eqmap}); in fact this follows from a stronger result (Theorem \ref{thm:eqmapmorita}). The general idea is to consider a Morita equivalent representation of $\tau$ (denoted $\Tau$) on a matrix algebra and study its connection with $\trho$. This approach provides more general and interesting results.
\subsection{Equivariant map between $\mathcal{E}$ and $\mathcal{R}$.}\label{sec:QPA:subsec:Equivariantmap} 
Let $(D,A_0)$ be an extended Hecke datum, $\Gamma$ be a $D$-graph with vertex set $V$, edge set $E$ and set of paths $P$ (in this section we do not use the superscripts $\Gamma$ as we do not consider any interaction between two distinct $D$-graphs). Let $p=(x_0 \xleftarrow{e_1} x_1 \cdots \xleftarrow{e_{n}} x_n) \in P$.  We extend $\mu: P \rightarrow A$ by setting $\mu(p):= \mu(e_1) \cdots \mu(e_{n})$ for any path of length $n\geq 1$, and $\mu(x):=1_A$  for $x \in V$. Define the map $u_0: A_0[\Gamma] \rightarrow \mathcal{E}$ as 
 	$ u_0(p) =\mu(p) s(p)$ for $p \in P$,
and extend it linearly to any elements in $A_0[\Gamma]$. Since $A_0[\Gamma]$ is a free $A_0$-module on $P$ then $u_0$ is an $A_0$-module homomorphism. We apply Proposition \ref{prop:adjoint} (regarding $\mathcal{E}$ as $\Hom_A(A,\mathcal{E})$) to get an $A$-module homomorphism $u:A[\Gamma] \rightarrow \mathcal{E}$ ($u$ is the adjoint map of $u_0)$. One can check $u(ap) = au_0(p) = a \mu(p)s(p),$ for all $a\in A, p\in P$.
	\begin{theorem}\label{thm:eqmap}
	The following statements hold.
	\begin{enumerate}[(I)]
		\item \label{property:mapufactors}The maps $u_0$ and $u$ factor through $\mathcal{R}_0$ and $\mathcal{R}$ respectively i.e., there exist a unique $A_0$-module homomorphism $\tilde{u}_0$ and a unique $A$-module homomorphism $\tilde{u}$ such that the following diagrams commute
		\begin{equation}\label{eq:uandtau}
		\begin{tikzpicture}[description/.style={fill=white,inner sep=2pt},baseline=(current  bounding  box.center)]
			\matrix (m) [matrix of math nodes, row sep=3em, column sep=2.5em, text height=1.5ex, text depth=0.25ex]
			{ A_0[\Gamma] & & \mathcal{E} \\
			& \mathcal{R}_0 & \\ };
			%\draw[double,double distance=5pt] (m-1-1) – (m-1-3);
			\path[->,font=\scriptsize]	
			(m-1-1) edge node[auto] {$ u_0$} (m-1-3)
				  edge node[auto,left] {$ \pi_0 $} (m-2-2)
			(m-2-2) edge node[auto,right] {$  \tilde{u}_0 $} (m-1-3);
		\end{tikzpicture}
		\begin{tikzpicture}[description/.style={fill=white,inner sep=2pt},baseline=(current  bounding  box.center)]
			\matrix (m) [matrix of math nodes, row sep=3em, column sep=2.5em, text height=1.5ex, text depth=0.25ex]
			{ A[\Gamma] & & \mathcal{E} \\
			& \mathcal{R} & \\ };
			%\draw[double,double distance=5pt] (m-1-1) – (m-1-3);
			\path[->,font=\scriptsize]	
			(m-1-1) edge node[auto] {$ u$} (m-1-3)
				  edge node[auto,left] {$ \pi $} (m-2-2)
			(m-2-2) edge node[auto,right] {$  \tilde{u} $} (m-1-3);
		\end{tikzpicture}
		\end{equation}
where $\pi,\pi_0$ are the natural quotient maps. Moreover, identifying $\mathcal{E}$ as $\allowbreak\Hom_A( A,\mathcal{E})$, the map $\tilde{u}$ is the adjoint map of $\tilde{u}_0$ as in Proposition \ref{prop:adjoint}.
		\item \label{property:mapuequiv}The map $\tilde{u}$ is an equivariant map (i.e., an $\mathcal{H}(D)$-module homomorphism) between  $\mathcal{R}$ and $\mathcal{E}$.
		\end{enumerate}
	\end{theorem} 
The proof of the theorem will follow from Theorem \ref{thm:eqmapmorita}.
%%%%%%%%%%%%%%%%%%%%%%%%%%%%%%%%%%%%
\subsection{Morita Equivalence}\label{subsec:morita}

In what follows, let $(D, A_0)$ be an extended Hecke datum, let $\Gamma$ be a $D$-graph with vertex set $V$. For any two sets $X,Y$ let $A^{X,Y}$ be the set of matrices with entries in $A$, indexed by $X\times Y$, such that all but finitely many entries are zero. For any $x\in X, y \in Y$, we denote by $e_{x,y}$ the matrix with $1_A$ in the $(x,y)$ entry and zero elsewhere. The set $\{e_{x,y}|x\in X, y\in Y\}$ forms a basis for $A^{X,Y}$ as an $A$-module. If $Y=\{y\}$ we denote $A^{X,Y}$ as $A^{X,\cdot}$, and $e_{x,y}$ as $e_x$. If $X=Y$ then $A^{X,X}$ is an $A$-algebra with enough orthogonal idempotents given by $\{e_{x,x}\}_{x\in X}$.

 Given a ring $R$ with enough orthogonal idempotents $\{e_x\}_{x\in I}$ we say an $R$-module $M$ is diagonalizable if  $M=\oplus_{x\in I} e_xM$. We claim the category of left $A$-modules (denoted $A$-Mod) is equivalent to the category of diagonalizable left $A^{V,V}$-modules (denoted $A^{V,V}$-Mod) for the (possibly infinite) set $V$. The proof of this fact follows from the proof of Theorem 3.54 in \cite{CurtisMorita}, which can be extended to the pair of rings $A, A^{V,V}$ using the fact that both are rings with enough orthogonal idempotents. The equivalence is given by the pair of functors $(F,G)$ where $F: A\text{-Mod}\rightarrow A^{V,V}\text{-Mod}$ such that $F(M)=A^{V,\cdot} \otimes_A M$, and $G:A^{V,V}\text{-Mod}\rightarrow A\text{-Mod}$ such that $G(N)=A^{\cdot,V}\otimes_{A^{V,V}} N$. Thus, the image of the free $A$-module (on the set $V$) $\mathcal{E}$ under the equivalence of categories is the $A^{V,V}$-module $A^{V,\cdot} \otimes_A \mathcal{E}$.
 
In Section \ref{subsec:IHAD-definition} we defined a representation $\tau$ of $\mathcal{H}(D)$ on $\mathcal{E}$. The associated Morita equivalent representation $\Tau$ of $\mathcal{H}(D)$ on $A^{V, \cdot} \otimes_A \mathcal{E}$ is given by 
 	\begin{align*}
		\Tau: \mathcal{H}(D) &\rightarrow End_{A^{V,V}}(A^{V, \cdot} \otimes_A \mathcal{E})\\		
			h&\mapsto \Tau(h):= \id_{A^{V, \cdot}} \otimes \tau(h) \qquad \text{ i.e., }\Tau(h)(e_x\otimes y) = e_x \otimes \tau(h)(y).
	\end{align*} 
 For any $r\in S$, we denote $\Tau_r$ the map $\id_{A^{V,\cdot}}\otimes \tau_r$. Since $\mathcal{E}$ is a free $A$-module on the set $V$, there is a canonical $A^{V,V}$-module isomorphism
	\begin{equation}\label{isomorita}
	 A^{V, \cdot} \otimes_A \mathcal{E} \rightarrow A^{V,V} \text{ with } e_x\otimes y \mapsto e_{x,y}.
 	\end{equation}
 Throughout this section we use $(\ref{isomorita})$ to identify $A^{V,V}$ with $A^{V,\cdot}\otimes_A\mathcal{E}$. We will later exploit the fact that $A^{V,V}$ has a natural $A$-algebra structure given by matrix multiplication, where as $A^{V, \cdot} \otimes_A \mathcal{E}$ is only considered as an $A^{V,V}$-module.
 
Define the map $U_0 : V^\Gamma \cup E^\Gamma \rightarrow A^{V,V}$ such that 
	\[x\in V^\Gamma \mapsto e_{x,x} \ \ \ \tand  \ \ \ (x\xleftarrow{e}y) \in E^\Gamma \mapsto \mu(e)e_{x,y}.\] 
%	\[   U_0(p)=\mu(p) e_{t(p), s(p)} \quad \text{for all } p \in P.\]
Using the presentation of $A_0[\Gamma]$ in Remark \ref{rmk:altPA}, one can quickly verify that $U_0$ is an $A_0$-algebra homomorphism. 
  Similar to what we did in Section \ref{sec:QPA:subsec:Equivariantmap} , we identify $A^{V,V}$ with $\Hom_A(A,A^{V,V})$ and use Proposition \ref{prop:adjoint} to get an $A$-module homomorphism $U: A[\Gamma] \rightarrow A^{V,V}$ adjoint to $U_0$. It follows that $U(ap)=aU_0(p)=a\mu(p) e_{t(p), s(p)}$ for all $a\in A, p \in P$, thus $U$ is an $A$-algebra homomorphism. We are ready to state the main theorem of this section.
	\begin{theorem}\label{thm:eqmapmorita}
	Given $A^{V,V}, A_0[\Gamma]$ and $A[\Gamma]$ as above, the following statements hold.
	\begin{enumerate}[(I)]
		\item \label{property:Ufactor}  The map $U_0$ factors through $\mathcal{R}_0$ and and $U$ factors through $\mathcal{R}$, i.e., there exists an $A_0$-algebra (resp. A-algebra) homomorphism $\tilde{U}_0$ (resp. $\tilde{U}$) such that the following diagrams commute
		\begin{equation}\label{diag:uandtaumorita}
		\begin{tikzpicture}[description/.style={fill=white,inner sep=2pt},baseline=(current  bounding  box.center)]
			\matrix (m) [matrix of math nodes, row sep=2em, column sep=1.5em, text height=1.5ex, text depth=0.25ex]
			{ A_0[\Gamma] & & A^{V,V} \\
			& \mathcal{R}_0 & \\ };
			%\draw[double,double distance=5pt] (m-1-1) – (m-1-3);
			\path[->,font=\scriptsize]	
			(m-1-1) edge node[auto] {$ U_0$} (m-1-3)
				  edge node[auto,left] {$ \pi_0 $} (m-2-2)
			(m-2-2) edge node[auto,right] {$  \tilde{U}_0 $} (m-1-3);
		\end{tikzpicture}
		\begin{tikzpicture}[description/.style={fill=white,inner sep=2pt},baseline=(current  bounding  box.center)]
			\matrix (m) [matrix of math nodes, row sep=2em, column sep=1.5em, text height=1.5ex, text depth=0.25ex]
			{ A[\Gamma] & & A^{V,V} \\
			& \mathcal{R} & \\ };
			%\draw[double,double distance=5pt] (m-1-1) – (m-1-3);
			\path[->,font=\scriptsize]	
			(m-1-1) edge node[auto] {$ U$} (m-1-3)
				  edge node[auto,left] {$ \pi $} (m-2-2)
			(m-2-2) edge node[auto,right] {$  \tilde{U} $} (m-1-3);
		\end{tikzpicture}
		\end{equation}
where $\pi,\pi_0$ are the natural quotient maps. Moreover, regarding $A^{V,V}$ as $\Hom_A(A,A^{V,V})$, the map $\tilde{U}$ is the adjoint map of $\tilde{U}_0$ as in Proposition \ref{prop:adjoint}. \label{property:Uopi}
		\item \label{property:Uequiv}The map $\tilde{U}$ is an equivariant map (i.e, an $\mathcal{H}(D)$-module homomorphism) between $\mathcal{R}$ and $A^{V,V}$.
		\item  \label{property:Urestric} For any $x \in V$  the submodules $[x]\mathcal{R}$ and $A^{\{x\},V}$ (or $Ae_x \otimes_A \mathcal{E}$ under the identification in $(\ref{isomorita})$) are invariant submodules of  $\mathcal{R}$ and $A^{V,V}$ respectively. Moreover, the restriction of $\tilde{U}$ on $[x]\mathcal{R}$ (denoted $\tilde{U}_x$) is an  $\mathcal{H}(D)$-module homomorphism between $[x]\mathcal{R}$ and $A^{\{x\},V}$.  
		\end{enumerate}
	\end{theorem} 
	We need the following lemma to prove the theorem.
	\begin{lemma}\label{lemma:utaumorita}
		Let $q \in A[\Gamma]$, $r_1, \dots, r_n \in S$  then
			\[ U\left( \rho_{r_1} \rho_{r_2} \cdots \rho_{r_n} (q) \right) = \mu(q)(e_{t(q)}\otimes\tau_{r_1}\tau_{r_2}\cdots\tau_{r_n}(s(q))).\]
		Moreover, for $h \in \mathcal{H}(D)$ we have $U\left( \rho(h) (q) \right) =\Tau (h)(U(q)).$
	\end{lemma}
	\begin{proof}
		In this proof we make use of the identification between $A^{V, V}$ and $A^{V, \cdot} \otimes_A \mathcal{E}$ in (\ref{isomorita}). We will prove this by induction on $n$. Let $n=1$ and assume $p \in P$, then  
	\small{	\begin{align*} 
		U(\rho_{r_1} (p)) 
				&= \begin{cases}
			a_{r_1} \mu(p)(e_{t(p)} \otimes s(p)) + \sum\limits_{\substack{y,e|r_1 \in \mathcal{L}(y)\\ (s(p) \xleftarrow{e} y) \in E}} \mu(p)\mu(e)(e_{t(p)}\otimes y) & \text{ if $r_1 \not\in \mathcal{L}(x)$}\\
				-b_{r_1} \mu(p)(e_{t(p)} \otimes s(p)) & \text{ if $r_1 \in \mathcal{L}(x)$}\end{cases}\\
				&=e_{t(p)} \otimes \mu(p)\tau_{r_1}(s(p))=(\id_{A^{V,\cdot}}\otimes \tau_{r_1})(e_{t(p)}\otimes \mu(p)s(p))=\Tau_{r_1}(U(p)).
	\end{align*}}
Now let $q\in A[\Gamma]$, then $q=\sum_{p}a_{p}p$ where $a_{p}\in A, p\in P$. By linearity, 			
		\[U\bigg(\rho_{r_1}\bigg(\sum_{p}a_{p}p\bigg)\bigg)=\sum_{p}U(\rho_{r_1}(a_{p}p))=\sum_{p}\Tau_{r_1}(U(a_{p}p))=\Tau_{r_1}\bigg(U\bigg( \sum_{p}a_{p}p\bigg)\bigg).\]
This completes the $n=1$ case. The general case follows by induction. The last statement follows from the fact that $\mathcal{H}(D)$ is freely generated (as an $A$-module) by $\{\T_{w}| w\in W\}$. 	
	\end{proof}
%%%%%%%%%%%%%%%%%%%%%%%%%%%
Before proceeding to the proof of the theorem, we remark that both Theorem \ref{thm:eqmapmorita} and Lemma \ref{lemma:utaumorita} are special cases of a more general result presented in Section \ref{sec:WGONCA} (Theorem \ref{thm:universality}).
\begin{proof}[Proof of Theorem \ref{thm:eqmapmorita}]
		To prove $(\ref{property:Uopi})$ we first show $U_0(J_0)=0$. Recall that $J_0$ is the two-sided ideal of $A_0[\Gamma]$ generated by elements of the form $\sum_{p} X_{p,i}^{r,s,x}p$ where  
	\begin{equation}\label{eq:pfthmU}
		(\rho_r \rho_s \rho_r \dots)_{m_{r,s}}(x)-(\rho_s \rho_r \rho_s \dots)_{m_{r,s}} (x) = \sum_{i,p} c_iX_{p,i}^{r,s,x}  p 
	\end{equation}
	and we do this for every pair $(r,s) \in S^2_{fin}$ and every $x \in V$.
	Let $LHS$ (resp. $RHS$) be the left hand side (resp. right hand side) of Eq. (\ref{eq:pfthmU}). By Lemma \ref{lemma:utaumorita} we get 
	\[U(LHS)= e_{x} \otimes (\tau_r \tau_s \tau_r \dots)_{m_{r,s}}(x)-(\tau_s \tau_r \tau_s \dots)_{m_{r,s}} (x)=0,\] therefore $U(RHS)=0$ or equivalently $\sum_{i,p}c_i U_0(X_{p,i}^{r,s,x} p)=0$. 
		Since $A^{V,V}$ is a free $A_0$-module (it is a free $A$-module and $A$ is free over $A_0$), then for any fixed $i \in I$ we get $U_0(\sum_{p} X_{p,i}^{r,s,x} p)=0$. Since $U_0$ is an $A_0$-algebra homomorphism, then $U_0(J_0)=0$ and we get the map $\tilde{U}_0$ as desired. 
	
	To prove $U(J)=0$, let $a \in A, j_0\in J_0$, so $aj_0 \in J=A\otimes_{A_0}J_0$, then $U(aj_0) = aU_0 (j_0)=0$ and thus we get the map $\tilde{U}$. To prove the last statement in $(\ref{property:Uopi})$ let $\hat{U}$ be the adjoint map to $\tilde{U}_0$ as in Proposition \ref{prop:adjoint}. We wish to show $\hat{U}=\tilde{U}$. By definition of the adjoint map, one can check $\hat{U}(a[p])=a\tilde{U}_0([p])$. By the left diagram in $(\ref{diag:uandtaumorita})$ (just proven) and the definition of $U$, we have $a\tilde{U}_0([p])=aU_0(p)= U(ap)$. By the right diagram in $(\ref{diag:uandtaumorita})$ we have $U(ap)=\tilde{U}(a[p])$. Thus $\hat{U}=\tilde{U}$.
		%%%%%%%%%%
		%%%%%%%
		
	To prove $(\ref{property:Uequiv})$  let $h \in \mathcal{H}(D)$, then subdiagrams $2$ and $3$ in $(\ref{diag:UandTau})$ commute by Theorem \ref{thm:eqmapmorita}($\ref{property:Uopi}$) and Lemma \ref{lemma:utaumorita} respectively. Subdiagram 1 commutes by definitions of $\rho$ and $\trho$. Since $\pi$ is surjective, the full diagram commutes too. 
	 \begin{equation}\label{diag:UandTau}
		\begin{tikzpicture}[baseline=(current  bounding  box.center)]
			\matrix(a)[matrix of math nodes,
			row sep=2em, column sep=2em,
			text height=1.5ex, text depth=0.25ex]
			{\mathcal{R}&&&\mathcal{R}\\
			&A[\Gamma]&A[\Gamma]&\\
			A^{V,V}&&&A^{V,V}\\};
			\path[->](a-2-2) edge node[left]{$\pi$}(a-1-1);
			\path[->](a-1-1)	edge node[ right] {$ \quad2 $} (a-3-1);
			\path[->](a-2-2)	edge node[ right] {$ \qquad\qquad\;\; 3 $} (a-3-1);
			\path[->](a-1-4)	edge node[ left] {$ 2\quad $} (a-3-4);
			\path[->](a-2-2)	edge node[ right] {$ \qquad\qquad\;\; 1 $} (a-1-1);
			\path[->](a-2-3) edge node[right]{$\pi$}(a-1-4);
			\path[->](a-1-1) edge node[left]{$\tilde{U}$}(a-3-1);
			\path[->](a-1-4) edge node[right]{$\tilde{U}$}(a-3-4);
			\path[->](a-1-1) edge node[above]{$\trho(h)$}(a-1-4);
			\path[->](a-2-2) edge node[above]{$\rho(h)$}(a-2-3);
			\path[->](a-2-2) edge node[above]{$U$}(a-3-1);
			\path[->](a-2-3) edge node[above]{$U$}(a-3-4);
			\path[->](a-3-1) edge node[above]{$\mathcal{T}(h)$}(a-3-4);
		\end{tikzpicture}
	\end{equation}
	To prove $(\ref{property:Urestric})$, by definition of $\trho$ and $\Tau$ we have $\trho_{r}([x]\mathcal{R})\subseteq [x]\mathcal{R}$ and $\Tau_{r}(A^{\{x\},V})\subseteq A^{\{x\},V}$ for all $r\in S$. It then follows that both submodules are invariant under the action of $\mathcal{H}(D)$. The fact that $im(\tilde{U}_x) \subseteq A^{\{x\},V}$ follows from the definition of $\tilde{U}$. Since $\tilde{U}$ is an $\mathcal{H}(D)$-module homomorphism, so is any restriction of $\tilde{U}$ on an invariant submodule.
	\end{proof}
	As stated before, Theorem \ref{thm:eqmap} follows as a corollary to Theorem \ref{thm:eqmapmorita}.
\begin{proof}[Proof of Theorem \ref{thm:eqmap}]
To prove part $(\ref{property:mapufactors})$, let $\pi_2: A^{V\times V}\rightarrow \mathcal{E}$ be the $A$-module homomorphism such that $e_{x,y} \mapsto y$, then $\pi_2(U_0(p))=\mu(p)s(p)=u_0(p)$ and $\pi_2(\tilde{U}_0([p]))=\mu(p)s(p)=\tilde{u}_0([p])$. By Theorem \ref{thm:eqmapmorita}$(\ref{property:Ufactor})$ we have $\tilde{U}_0\circ\pi_0=U_0$, thus the left diagram in $(\ref{dig:proofeqmap})$ commutes. Similarly, $\pi_2\circ U=u, \pi_2\circ\tilde{U}=\tilde{u}$ and $\tilde{U}\circ\pi=U$, thus the right diagram in $(\ref{dig:proofeqmap})$ commutes. Let $\hat{u}$ be the map adjoint to $\tilde{u}_0$ as in Proposition \ref{prop:adjoint}. Similar to Theorem \ref{thm:eqmapmorita}$(\ref{property:Ufactor})$, one can check $\hat{u}(a[p])=a\tilde{u}_0([p])=\tilde{u}(a[p])$.
	\begin{equation}\label{dig:proofeqmap}
		\begin{tikzpicture}[description/.style={fill=white,inner sep=2pt},baseline=(current  bounding  box.center)]
			\matrix (m) [matrix of math nodes, row sep=2.5em, column sep=3em, text height=1.5ex, text depth=0.25ex]
			{ A_0[\Gamma] & & \mathcal{E} \\
			& A^{V,V} & \\ 
			& \mathcal{R}_0 & \\ };
			%\draw[double,double distance=5pt] (m-1-1) – (m-1-3);
			\path[->,font=\scriptsize]	
			(m-1-1) edge node[auto] {$ u_0$} (m-1-3)
				  edge node[auto,left] {$ \pi_0 $} (m-3-2)
				  edge node[auto,right] {$ U_0 $} (m-2-2)
			(m-2-2) edge node[auto,left] {$  \pi_2 $} (m-1-3)
			(m-3-2) edge node[auto,left] {$  \tilde{U}_0 $} (m-2-2)
			(m-3-2) edge node[auto,right] {$  \tilde{u}_0 $} (m-1-3);
		\end{tikzpicture}
		\begin{tikzpicture}[description/.style={fill=white,inner sep=2pt},baseline=(current  bounding  box.center)]
			\matrix (m) [matrix of math nodes, row sep=2.5em, column sep=3em, text height=1.5ex, text depth=0.25ex]
			{ A[\Gamma] & & \mathcal{E} \\
			& A^{V,V} & \\ 
			& \mathcal{R}& \\ };
			%\draw[double,double distance=5pt] (m-1-1) – (m-1-3);
			\path[->,font=\scriptsize]	
			(m-1-1) edge node[auto] {$ u$} (m-1-3)
				  edge node[auto,left] {$ \pi $} (m-3-2)
				  edge node[auto,right] {$ U $} (m-2-2)
			(m-2-2) edge node[auto,left] {$  \pi_2 $} (m-1-3)
			(m-3-2) edge node[auto,left] {$  \tilde{U} $} (m-2-2)
			(m-3-2) edge node[auto,right] {$  \tilde{u} $} (m-1-3);
		\end{tikzpicture}
	\end{equation}

To prove Theorem \ref{thm:eqmap}$(\ref{property:mapuequiv})$ notice that subdiagrams 2 in (\ref{diag:u}) commute as mentioned in the proof of Theorem \ref{thm:eqmap}$(\ref{property:mapufactors})$. Subdiagram 1 commutes by Theorem \ref{thm:eqmapmorita}$(\ref{property:Uequiv})$. To show that subdiagram 3 commutes recall that $\Tau = \id_{A^{V,\cdot}}\otimes \tau$ (identifying $A^{V,V}$ with $A^{V,\cdot}\otimes_A \mathcal{E}$). Thus when taking the projection onto the second component, we get $\pi_2\Tau=\tau$. It now follows that the full diagram in (\ref{diag:u}) commutes.
	\begin{equation}\label{diag:u}
		\begin{tikzpicture}[>=angle 90, baseline=(current  bounding  box.center)]
			\matrix(a)[matrix of math nodes,
			row sep=2em, column sep=2em,
			text height=1.5ex, text depth=0.25ex]
			{\mathcal{R}&&&\mathcal{R}\\
			&A^{V,V}&A^{V,V}\\
			\mathcal{E}&&&\mathcal{E}\\};
			\path[->](a-1-1) edge node[above]{$\trho(h)$}(a-1-4);
			\path[->](a-1-1) edge node[left]{$\tilde{u}$}(a-3-1);
			\path[->](a-1-1)	edge node[ right] {$ \quad2 $} (a-3-1);
			\path[->](a-2-2)	edge node[ right] {$ \qquad\qquad\;\; 3 $} (a-3-1);
			\path[->](a-1-4)	edge node[ left] {$ 2\quad $} (a-3-4);
			\path[->](a-1-1)	edge node[ right] {$ \qquad\qquad\;\; 1 $} (a-2-2);
			\path[->](a-2-2) edge node[above]{$\mathcal{T}(h)$}(a-2-3);
			\path[->](a-1-4) edge node[right]{$\tilde{u}$}(a-3-4);
			\path[->](a-1-4) edge node[left]{$\tilde{U}$} (a-2-3);
			\path[->](a-1-1) edge node[right]{$\tilde{U}$} (a-2-2);
			\path[->](a-2-2) edge node[right]{$\pi_2$} (a-3-1);
			\path[->](a-2-3) edge node[left]{$\pi_2$} (a-3-4);
			\path[->](a-3-1) edge node[above]{$\tau(h)$} (a-3-4);
		\end{tikzpicture}\vspace{-3mm}
	\end{equation}
	\end{proof}

	\begin{corollary}\label{cor:nontrivialalgebra}
If $\Gamma$ is a non-empty $D$-graph, then the elements in $\{[x]\}_{x\in V}$ are distinct non-zero orthogonal idempotents in $\mathcal{R}$. Moreover, for any $p \in P$ with $\mu(p)\neq 0$ we have $[p]\neq0 \in \mathcal{R}$.
	\end{corollary}					
	\begin{proof}
	By Theorem \ref{thm:eqmapmorita} we have $\tilde{U}([x])=e_{x,x}\neq 0$ for all $x \in V$, hence $[x] \neq 0 \in \mathcal{R}$. Similarly, for any $p\in P$ such that $\mu(p)\neq 0$ we have $\tilde{U}([p]) = \mu(p)e_{t(p),s(p)} \neq 0$, hence $[p] \neq 0$.
	\end{proof} 
We say that a $D$-graph $\Gamma$ is connected if $(V,E)$ is a strongly connected quiver and $\mu(e)\neq 0 \in A$ for all $e \in E$. In what follows, suppose $A_0\subseteq A$ is a domain with field of fractions $Q$. Given a $D$-graph $\Gamma$,  let $A_Q :=A\otimes_{A_0}Q$. Since $A$ is a free $A_0$-module with some basis $\{c_i\}_{i\in I}$ then $A\cong\bigoplus_{i \in I} A_0 \cdot c_i$ and $A_Q\cong \bigoplus_{i \in I} Q \cdot c_i $  is a free $Q$-module with the same basis. Let $D_Q=(W,S,A_Q,(a_r)_{r\in S}, (b_r)_{r\in S})$, then $(D_Q,Q)$ is an extended Hecke datum and $\Gamma$ is a $D_Q$-graph.
	\begin{corollary}\label{cor:Uxsurjective}
	Let $\Gamma$ be a $D_Q$-graph with $\mu(e)\in A_0\subseteq A$ for all $e\in E$.
	\begin{enumerate}[(I)]
		\item \label{property:tauandux}If $\Gamma$ is  connected or if $\mathcal{E}$ is irreducible under the representation $\tau$, then $\tilde{U}_x$ is surjective for all $x \in V$ and  the representation $\tau$ on $\mathcal{E}$ is isomorphic to the restriction of $\trho$ on $[x]\mathcal{R}/ ker(\tilde{U}_x)$.
			\item \label{property:wgraph} Let $\Gamma$ be the $D_Q$-graph associated to a $W$-cell (i.e., strongly connected component of a $W$-graph, see \cite{JS}) then for each $x \in V$ the representation of $\mathcal{H}(D_Q)$ on $[x]\mathcal{R}/ ker(\tilde{U}_x)$ is a copy of the $\tau$ representation.
	\end{enumerate}
		\end{corollary}
		\begin{proof}
By Theorem \ref{thm:eqmapmorita}$(\ref{property:Urestric})$ we have that $\tilde{U}_x$ is an (injective) equivariant map between $[x]\mathcal{R}/ ker(\tilde{U}_x)$ and $A^{\{x\},V}$, with respect to the representations $\trho$ and $\Tau$ respectively. Since $\Tau$ restricted to $A^{\{x\},V}$ is isomorphic to $\tau$, then to prove $(\ref{property:tauandux})$ it suffices to show that $\tilde{U}_x$ is surjective.

If $\Gamma$ is connected then for any $x,y \in V$ there is a path $p \in P$ with $t(p)=x, s(p)=y$ and $\mu(p)\neq 0$. Thus, $\tilde{U}_x(p/\mu(p))= \mu(p)e_{x,y}/\mu(p) = e_{x,y}$, therefore $\tilde{U}_x$ is surjective.  If $\mathcal{E}$ is irreducible (under $\tau$) then $A^{\{x\},V}$ is irreducible (under $\Tau$). Since $im(\tilde{U}_x)$ is a non-trivial invariant submodule of $A^{\{x\},V}$ then $im(\tilde{U}_x)= A^{\{x\},V}$, thus $\tilde{U}_x$ is surjective. Statement $(\ref{property:wgraph})$ is a special case of $(\ref{property:tauandux})$.
		\end{proof}
%%%%%%%%%%%%%
%%%%%%%%%%%%%
%%%%%%%%%%%%%
\section{Computing the generators of $J_0$}\label{universalPDG}
In this section we use a collection of graphs (called universal pre-$D$-graphs) to help compute the generators of the ideal $J_0^{\Lambda}$ for any finite pre-$D$-graph $\Lambda$ (i.e., with finitely many vertices and edges). The main idea is to compute the generators for the universal pre-$D$-graphs and the images of these generators (under a suitable map) will form a generating set for $J_0^\Lambda$. We show several explicit computations for the extended Hecke datum $(D_{\mathbb{Z}},\mathbb{Z})$.

\subsection{Universal pre-$D$-graph}\label{subsec:universal}
Let $(D,A_0)$ be an extended Hecke datum, $(r,s) \in S^2_{fin}$, and let $\Gamma^{\mathcal{U}}$ be the complete quiver on $V^{\Gamma^{\mathcal{U}}}:=\{x_r, x_s,x_\varnothing,x_\wp\}$. We now define a map $\mathcal{L}^{\Gamma^{\mathcal{U}}}$ that makes $\Gamma^{\mathcal{U}}$ a pre-$D$-graph. Let $\mathcal{L}^{\Gamma^{\mathcal{U}}}$ be such that $\mathcal{L}^{\Gamma^{\mathcal{U}}}(x_r) := \{r\},\mathcal{L}^{\Gamma^{\mathcal{U}}}(x_s) := \{s\},\mathcal{L}^{\Gamma^{\mathcal{U}}}(x_{\varnothing}) := \emptyset$ and $\mathcal{L}^{\Gamma^{\mathcal{U}}}(x_\wp) := \{r,s\}$  (see Figure \ref{fig:univPDG}). We say $\Gamma^\mathcal{U}$ is the \textbf{universal pre-$D$-graph with respect to $r,s$}. 
	\begin{figure}
		\begin{center}
			\begin{tikzpicture}[->,node distance=3cm,scale=0.6, every node/.style={scale=0.6}]
 				 \node[draw,circle]  (1) [label=above left:\LARGE{$x_\wp$}]{$\{ r,s \}$};
				  \node[draw,circle]  (2) [below left of=1,label=below left:\LARGE{$x_r$}] {$\{r\}$};
				  \node[draw,circle]  (3) [below right of=1,label=below right:\LARGE{$x_s$}] {$\{s\}$};
 				 \node [draw,circle] (4) [below right of=2,label=below left:\LARGE{$x_{\varnothing}$}] {$\emptyset $};
				  \path[every node/.style={font=\sffamily\small}]
				    (1) edge [bend left] node {}(3)
				        	edge [bend left]node {} (4)
				 	edge [bend right] node {} (2)
					edge [loop above] node {} (1)
				(2) edge node {}(1)
			        		edge [bend right]node {} (4) edge  node {} (3)
					edge [loop left] node {} (2)
				(3) edge [bend left] node {}(2)
				        edge [bend left]node {} (4)
 					edge node {} (1)
					edge [loop right] node {} (1)
				(4) edge node {}(3)
   				     	edge node {} (1)
 					edge node {} (2)
					edge [loop below] node {} (4);
			\end{tikzpicture}
		\end{center}
	\caption{Universal pre-$D$-graph $\Gamma^{\mathcal{U}}$ with respect to $r,s$.}\label{fig:univPDG}
	\end{figure}
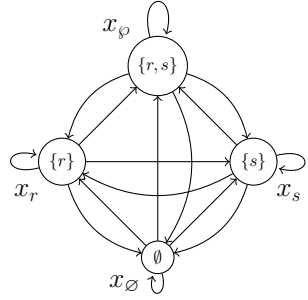

Let $\Lambda$ be any pre-$D$-graph with finite vertex set and finitely many edges, $r, s \in S$, $x\in V^{\Gamma^{\mathcal{U}}}$ and $V^{\Lambda}_x := \{y \in V^{\Lambda} \,|\, \mathcal{L}^{\Lambda}(y) \cap \{r,s\} = \mathcal{L}^{\Gamma^{\mathcal{U}}}(x)\} \subseteq V^{\Lambda}$. Let $\psi_{r,s}$ be the map
\begin{align*}
\psi_{r,s}:  V^{\Gamma^\mathcal{U}}\cup E^{\Gamma^\mathcal{U}} &\longrightarrow A[\Lambda] \\
 x &\mapsto \quad\sum y \quad\text{ where } y \in V^{\Lambda}_x\\
x_1 \xleftarrow{e} x_2 &\mapsto \quad \sum f \quad \text{ where } y_1\xleftarrow{f} y_2 \in E^{\Lambda},\; y_1 \in V^{\Lambda}_{x_1}, \; y_2 \in V^{\Lambda}_{x_2}
\end{align*}
for $x \in V^{\Gamma^{\mathcal{U}}}, x_1 \xleftarrow{e} x_2 \in E^{\Gamma^{\mathcal{U}}}$.  Using the description of $A[\Gamma^{\mathcal{U}}]$ in Remark \ref{rmk:altPA} we can show $\psi_{r,s}$ is an $A$-algebra homomorphism. 

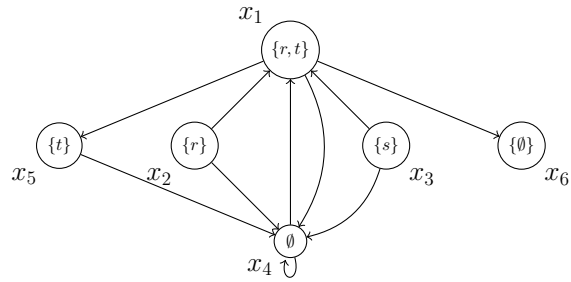
\begin{figure}
		\begin{center}
			\begin{tikzpicture}[->,node distance=3cm,scale=0.6, every node/.style={scale=0.6}]
 				 \node[draw,circle]  (1) [label=above left:\LARGE{$x_1$}]{$\{ r,t \}$};
				  \node[draw,circle]  (2) [below left of=1,label=below left:\LARGE{$x_2$}] {$\{r\}$};
				  \node[draw,circle]  (3) [below right of=1,label=below right:\LARGE{$x_3$}] {$\{s\}$};
 				 \node [draw,circle] (4) [below right of=2,label=below left:\LARGE{$x_4$}] {$\emptyset $};
				 \node [draw,circle](5) [left of=2,label=below left:\LARGE{$x_5$}]{$\{t\}$};
				  \node [draw,circle](6) [right of=3,label=below right:\LARGE{$x_6$}]{$\{\emptyset\}$};
				  \path[every node/.style={font=\sffamily\small}]
				    (1) edge [bend left] node {}(4)
					edge node {} (5)
					edge node {} (6)
				(2) edge node {}(1)
			        		edge node {} (4) 
				(3) edge [bend left] node {}(4)
 					edge node {} (1)
				(4) edge node {}(1)
					edge [loop below] node {} (4)
				(5) edge node {}(4);
			\end{tikzpicture}
		\end{center}
	\caption{An example of a finite pre-$D$-graph $\Lambda$}\label{fig:PDG}
	\end{figure}

\begin{example} Let $\Lambda$ be the pre-$D$-graph in Figure \ref{fig:PDG} (where the underlying Coxeter system of $D$ has at least three generators $r,s$ and $t$). The following tables show the image of $\psi_{r,s}$ for all vertices of $\Gamma^{\mathcal{U}}$ and some of its edges. 
\begin{center}
  \begin{tabular}{@{} |c|c| @{}}\hline
    vertex& $\psi_{r,s}(\text{vertex})$   \\ \hline
    $x_r$& $x_1+x_2$    \\ 
    $x_s$ & $x_3$\\ 
    $x_\wp$& 0  \\ 
    $x_\varnothing$ & $x_4+x_5+x_6$ \\ \hline
   \end{tabular}\qquad
      \begin{tabular}{@{} |c|c| @{}}\hline
    edge& $\psi_{r,s}(\text{edge})$  \\ \hline
    $x_\varnothing \leftarrow x_r$& $(x_5 \leftarrow x_1)+(x_6 \leftarrow x_1)+(x_4 \leftarrow x_2)$ \\ 
    $x_r \leftarrow x_s$ & $x_1 \leftarrow x_3$ \\ 
    $x_r\leftarrow x_\varnothing $& $x_1\leftarrow x_4$  \\ 
    $x_s\leftarrow x_\varnothing $ & 0  \\ 
 \hline
   \end{tabular}
%  \caption{TableCaption}
  \label{tab:label}
\end{center}
\end{example}
The next lemma is key in the proof of Theorem \ref{thm:psimap}.
\begin{lemma}\label{lemma:rhopsi}
Let $(r,s) \in S^2_{fin}$, and $q \in A[\Gamma^\mathcal{U}]$ then \small{$\rho^{\Lambda}_z(\psi_{r,s}(q))=\psi_{r,s}(\rho_z^{\Gamma^\mathcal{U}} (q))$} for $z \in \{r,s\}$.
		\end{lemma}

\begin{proof}
We show the result holds for $z=r$ and the case $z=s$ follows by symmetry. For $x \in V^{\Gamma}$, 
	\begin{align*}
		\psi_{r,s}(\rho_r^{\Gamma^\mathcal{U}}(x))&=\begin{cases}
				(a_r) \psi_{r,s}(x) + \psi_{r,s} (x \leftarrow x_r +x \leftarrow x_\wp)  & \text{ if $r \not\in \mathcal{L}^{\Gamma^\mathcal{U}}(x)$}\\
				(-b_r) \psi_{r,s}(x) & \text{ if $r \in \mathcal{L}^{\Gamma^\mathcal{U}}(x),$}\end{cases}\\
			&=\begin{cases}
				(a_r) \sum\limits_{y\in V^{\Lambda}_x} y + \sum\limits_{y\in V^{\Lambda}_x} y \sum\limits_{\substack{z, e | r \in \mathcal{L}^\Lambda (z) \\ (y\xleftarrow{e}z) \in E^\Lambda}} y \xleftarrow{e} z  & \text{ if $r \not\in \mathcal{L}^{\Gamma^\mathcal{U}}(x)$}\\
				(-b_r)  \sum\limits_{y\in V^{\Lambda}_x} y & \text{ if $r \in \mathcal{L}^{\Gamma^\mathcal{U}}(x),$}\end{cases}\\
		&= \rho_r^{\Lambda}(\psi_{r,s}(x)).
	\end{align*}
Since $\psi_{r,s}$ is an $A$-algebra homomorphism, then for $q \in P^{\Gamma^\mathcal{U}}$, 
	\begin{align*}
		\psi_{r,s}(\rho_r^{\Gamma^\mathcal{U}}(q))
%				&=\psi_{r,s}(q\rho_r^{\Gamma^\mathcal{U}}(s(q)))\\		
				&=\psi_{r,s}(q)\psi_{r,s}(\rho_r^{\Gamma^\mathcal{U}}(s(q)))=\psi_{r,s}(q) \rho_r^{\Lambda}(\psi_{r,s}(s(q)))
%				&= \rho_r^{\Lambda}(\psi_{r,s}(q) \psi_{r,s}(s(q)))\\
				= \rho_r^{\Lambda}(\psi_{r,s}(q)).
	\end{align*}
	Extend linearly to get the desired result.
\end{proof}

Let $\{\sum_{p} X_{p,i}^{r,s,x}p \,|\, x \in V^{{\Gamma^\mathcal{U}}}, i \in I, (r,s) \in S^2_{fin}\}$ be the generating set for $J_0^{\Gamma^{\mathcal{U}}}$ as in Eq. $(\ref{eq:relations})$. The next theorem states that computing the generators of $J_0^{\Lambda}$ for any pre-$D$-graph $\Lambda$ with finitely many vertices and finitely many edges can be done by computing the generators of $J_0^{\Gamma^{\mathcal{U}}}$ and then applying $\psi_{r,s}$. 
%%% THM
	\begin{theorem}\label{thm:psimap}
		For any $y \in V^{\Lambda}$, $(r, s) \in S^2_{fin}$ we get 
			\[ (\rho_r^{\Lambda} \rho^{\Lambda}_s \rho^\Lambda_r \cdots)_{m_{r,s}}(y) -(\rho^\Lambda_s \rho^\Lambda_r \rho^\Lambda_s \cdots)_{m_{r,s}} (y)=c_iy \psi_{r,s}\bigg(\sum_{i,p} X^{r,s,x_y}_{p,i} p\bigg),\]
where $x_y$ is the vertex in $V^{\Gamma^\mathcal{U}}$ with $\mathcal{L}^{\Gamma^\mathcal{U}} (x_y)=\mathcal{L}^\Lambda (y) \cap \{r,s\}$. Moreover, a generating set for $J_0^{\Lambda}$ is $\{y\sum_p X^{r,s,x_y}_{p,i}\psi_{r,s}(p)\,|\, (r,s)\in S^2_{fin}, i\in I, y \in V^\Lambda\}$.
%, i.e., $\sum_q X_{q,i}^{r,s,y}q =y\sum_p X^{r,s,x_y}_{p,i}\psi_{r,s}(p)$. 
	\end{theorem}
	\begin{proof}
		Let $m:=m_{r,s}$. Using the fact that $y=y\psi_{r,s}(x_y)\in A[\Lambda]$ and Lemma \ref{lemma:rhopsi} one can show 
\small{\[(\rho_r^{\Lambda} \rho^{\Lambda}_s \rho^\Lambda_r \cdots)_{m} (y) = y(\rho_r^{\Lambda} \rho^{\Lambda}_s \rho^\Lambda_r \cdots)_{m} (\psi_{r,s}(x_y))=y\psi_{r,s}((\rho_r^{\Gamma^\mathcal{U}} \rho_s^{\Gamma^\mathcal{U}} \rho^{\Gamma^\mathcal{U}}_r \cdots)_{m} (x_y)).\]}
We now compute the braid relations on $y$ to get 
	\begin{align*}
		(&\rho_r^{\Lambda} \rho^{\Lambda}_s \rho^\Lambda_r \cdots)_{m}(y) -(\rho^\Lambda_s \rho^\Lambda_r \rho^\Lambda_s \cdots)_{m} (y)\\
		&=(\rho_r^{\Lambda} \rho^{\Lambda}_s \rho^\Lambda_r \cdots)_{m}(y\psi_{r,s}(x_y)) -(\rho^\Lambda_s \rho^\Lambda_r \rho^\Lambda_s \cdots)_{m} (y\psi_{r,s}(x_y))\\
		&=y\psi_{r,s}\bigg(( \rho_r^{\Gamma^\mathcal{U}} \rho^{\Gamma^\mathcal{U}}_s \rho^{\Gamma^\mathcal{U}}_r \cdots)_{m}(x_y) -(\rho^{\Gamma^\mathcal{U}}_s \rho^{\Gamma^\mathcal{U}}_r \rho^{\Gamma^\mathcal{U}}_s \cdots)_{m} (x_y)\bigg)\\
		&=y \psi_{r,s}\bigg(\sum_{i,p} c_iX^{r,s,x_y}_{p,i} p\bigg).
	\end{align*}
Since $\psi_{r,s}$ is $A$-linear we get the result. The last statement follows directly by the definition of $J_0^{\Lambda}$.
	\end{proof}
\subsection{Extended Hecke datum $(D_{\mathbb{Z}}, \mathbb{Z})$}\label{sec:EIHDZ}
We now explicitly compute the generators for $J_0^{\Gamma^{\mathcal{U}}}$ in the special case where the extended Hecke datum is $(D_{\mathbb{Z}}, \mathbb{Z})$ (Theorem \ref{thm:relations}) and use Theorem \ref{thm:psimap} to compute the generators of $J_0^\Lambda$, for three interesting pre-$D$-graphs $\Lambda$ (see Examples \ref{subsec:onevertex},  \ref{subsec:B3} and  \ref{subsec:AHA}).
%%%%%

For the remainder of this section we work with the extended Hecke datum $(D_{\mathbb{Z}}, \mathbb{Z})$ where $D_{\mathbb{Z}}=(W,S,\allowbreak \mathbb{Z}[v,v^{-1}],\allowbreak v,v^{-1})$ (thus $A=\mathbb{Z}[v,v^{-1}]$) and let $\{v^i\}_{i \in \mathbb{Z}}$ be the basis for $\mathbb{Z}[v,v^{-1}]$ as a $\mathbb{Z}$-module. Throughout this section (unless otherwise noted) let $(r,s)\in S^2_{fin}$. In Section \ref{subsec:comput} we described an explicit formula to compute $(\rho_r \rho_s \rho_r \dots)_{m_{r,s}} - (\rho_s \rho_r \rho_s \dots)_{m_{r,s}}$ for any pre-$D$-graph. Here we will give a simplified version for when $(D,A_0)=(D_{\mathbb{Z}},\mathbb{Z})$. The main idea is to compute $(\dots \rho_r \rho_s \rho_r )_{m_{r,s}} - (\dots\rho_s \rho_r \rho_s)_{m_{r,s}}$ on the universal pre-$D_{\mathbb{Z}}$-graph (this computation differ by a power of $-1$ from one mentioned above).

Let $\Gamma^\mathcal{U}$ be the universal pre-$D_{\mathbb{Z}}$-graph with respect to $r,s$. Since there is only one (directed) edge from any one vertex to another, for any path $p=x_{\alpha_1}\xleftarrow{e_1} x_{\alpha_2} \cdots \xleftarrow{e_{n-1}} x_{\alpha_n} \in P^{\Gamma^{\mathcal{U}}}$ we denote $p$ as $\alpha_1\alpha_2\cdots \alpha_n$ for $\alpha_i \in \{ r,s,\varnothing,\wp\}$ (e.g. $x_s\leftarrow x_r\leftarrow x_\varnothing \leftarrow x_\wp$ is denoted as $sr\varnothing\wp$).

For $m$ a positive integer, let 
	\begin{equation}\label{eq:defcm}
		c_m(r,s) =\cm{m}{r}{s} \in \mathbb{Z}[\Gamma^{\mathcal{U}}],
	\end{equation}
	where $\lfloor \cdot \rfloor$ is the floor function on $\mathbb{Z}$. Let
\[ d_m(r,s)=\begin{cases}c_m(r,s) & \text{ if $m$ is even}
\\
c_m(s,r)& \text{ if $m$ is odd}\end{cases}, \qquad b_m=\begin{cases} s & \text{ if $m$ is even}
\\
r& \text{ if $m$ is odd}\end{cases}.\]
For example,
	\begin{align*}
	d_1(r,s) = r, \ \  d_2(r,s) = rs\ \, d_3(r,s) = rsr-r\ \ d_4(r,s) = rsrs-2rs\\
		d_5(r,s) = rsrsr-3rsr+r\ \ d_6(r,s) = rsrsrs-4rsrs+3rs.
	\end{align*}
	
%%%%%%%%%%%%%%%%%%%%%%%%%%%%%%%%%%%%
For any element $w\in W$ with a unique reduced expression $w=r_1\dots r_n$, let $\rho_{w}:=\rho_{r_1}\cdots \rho_{r_n}$. For any sequence $(r_1, \dots, r_n)$ with $r_i \in S$ for $i=1,\dots,n$ such that $w=r_1\cdots r_n$ is reduced and every $y \in W$ with $y<w$ (in Bruhat order) has a unique reduced expression, let $C_{r_1,\dots,r_n}$ be the endomorphism of $A[\Gamma^\mathcal{U}]$ defined as
 	\begin{equation}\label{eq:defrhoC}
		C_{r_1,\dots,r_n}:= \rho_{r_1}\cdots \rho_{r_n}+  \sum_{y < w} (-v)^{l(w)-l(y)}\overline{p_{y,w}} \rho_y,
	\end{equation}
 where $\rho_1=\id_{A[\Gamma^\mathcal{U}]}$, $p_{y,w}$ are the Kazhdan-Lusztig polynomials as defined in \cite{KL1}, and $\overline{p_{y,w}}:= p_{y,w}(v^{-1})$. 
If $w=r_1\dots r_n$ is a unique reduced expression for $w$ then we denote $C_{r_1,\dots,r_n}$ as $C_w$. It then follows
	\begin{equation}\label{eq:samerhoandC}	
		C_{(\dots r,s,r)_{m_{r,s}}}- C_{(\dots s,r,s)_{m_{r,s}}}= (\cdots \rho_r \rho_s \rho_r)_{m_{r,s}} - (\cdots \rho_s \rho_r \rho_s)_{m_{r,s}}.
	\end{equation}
	Since we will use the endomorphisms $C_r$ often, it is worth pointing out that $C_r=\rho_r - v \id_{A[\Gamma^\mathcal{U}]}$.
		\begin{lemma}\label{lemma:Cconditions}
Let $r,s \in S$, $\Gamma^\mathcal{U}$ be the universal pre-$D_{\mathbb{Z}}$-graph with respect to $r,s$, and  $m\in \mathbb{N}_{\geq2}$ such that $m\leq m_{r,s}$ then
		\begin{gather}
		 C_{(\dots r,s,r )_m}  = \sum_{j=0}^{\lfloor \frac{m-1}{2} \rfloor} D_{m,j} (\cdots C_r C_s C_r)_{m-2j}\label{eq:C1}\\
 		C_{(\dots s,r,s)_m} = \sum_{j=0}^{\lfloor \frac{m-1}{2} \rfloor} D_{m,j} (\cdots C_s C_r C_s )_{m-2j}\label{eq:C2},
		\end{gather}
where $D_{m,j} = (-1)^j \binom{m-1-j}{j}$.
	\end{lemma}
	\begin{proof}
	First we prove the result when $W$ is the universal Coxeter group on the set $S$. Let $\{\T_s| s \in S\}$ be the generating set of $\mathcal{H}(D_{\mathbb{Z}})$ as an $A$-algebra (recall $A=\mathbb{Z}[v,v^{-1}]$). The map $\phi : \mathcal{H}(D_{\mathbb{Z}}) \rightarrow End_{A[\Gamma^\mathcal{U}]}(A{[\Gamma^\mathcal{U}}])$ that sends $\T_s$ to $\rho_s$ for all $s \in S$ is a ring homomorphism since there are no braid relations in $\mathcal{H}(D_{\mathbb{Z}})$. Let $\{\C_w\}_{w\in W}$ be the Kazhdan-Lusztig basis of $\mathcal{H}(D_{\mathbb{Z}})$ (see \cite{L1} 6.6), then $\phi(\C_w)=C_w$ as defined in $(\ref{eq:defrhoC})$ (notice that $C_w$ makes sense as every element $w\in W$ has a unique reduced expression). We now show equations $(\ref{eq:C1})$ and $(\ref{eq:C2})$ are true for the Kazhdan-Lusztig basis and then apply $\phi$ to get the result for the $C_w$ endomorphisms. For $m=2$, it is known $\C_{sr}=\C_s\C_r$. Let $t=r$ if $m$ is odd and $t=s$ if $m$ is even. For $3\leq m\leq m_{r,s}$ it is known that $\C_t\C_{(\dots rsr)_{m-1}}=\C_{(\dots rsr)_{m}}+\C_{(\dots rsr)_{m-2}}$. Equation $(\ref{eq:C1})$  now follows by induction. The result for Eq. $(\ref{eq:C2})$ is obtained similarly.

Now let $W$ be any Coxeter group. The definition of $\rho_r$ and $\rho_s$ does not depend on the value of $m_{r,s}$ and hence $\rho_r, \rho_s$ are the same endomorphisms of $A[\Gamma^\mathcal{U}]$ as the ones defined for the universal Coxeter group. Thus, $C_r, C_s, C_{(\dots r,s,r )_m}$ and $C_{(\dots s,r,s)_m} $ are also the same endomorphisms as the ones defined for the universal Coxeter group, hence the formulas $(\ref{eq:C1})$ and $(\ref{eq:C2})$ still hold.
	\end{proof}

For the following lemma let $\beta=v+v^{-1}$.
	\begin{lemma}\label{lemma:Cronr}
	Let $r,s,\varnothing$ and $\wp$ be the vertices of $\Gamma^{\mathcal{U}}$  (see Figure \ref{fig:univPDG}), then 
		\begin{gather}
			(\cdots C_rC_sC_r)_{m}(r)=-\beta( rsr\cdots)_m + \sum_{i=1}^{m-1} (-\beta)^i ( rsr\cdots)_{m-i}\wp\label{eq:rsronr}\\
			(\cdots C_sC_rC_s)_{m}(r)=( rsr\cdots)_{m+1} + \sum_{i=0}^{m-1} (-\beta)^i ( rsr\cdots)_{m-i}\wp.\label{eq:srsonr}\\
			( \dots C_r C_s C_r)_m (\varnothing)= \varnothing(rsr \dots)_m + \sum_{i=0}^{m-1} (-\beta)^i \varnothing(rsr \dots)_{m-1-i} \wp \label{eq:rsront}\\
			( \dots C_s C_r C_s)_m (\varnothing)= \varnothing(srs \dots)_m + \sum_{i=0}^{m-1} (-\beta)^i \varnothing(srs \dots)_{m-1-i} \wp.\label{eq:srsont}
		\end{gather}
	To find formulas for the alternating endomorphisms above on the vertex $s$, interchange $r$ with $s$ in Equations (\ref{eq:rsronr}) and (\ref{eq:srsonr}). Finally,  for any $r_i \in \{r,s\}, i=1, 2 \dots, n$,  we get $C_{r_1}C_{r_2}\cdots C_{r_n} (\wp) = (-\beta)^n \wp$.
	\end{lemma}
	\begin{proof}
	We use induction to prove Eq. $(\ref{eq:rsronr})$. For $m=1$, $C_r(r)=-\beta r$, hence the result is true. Suppose the result holds for any $k<m$, then 
	\small{\begin{align*}
		C_{b_m}(\cdots C_rC_sC_r)_{m-1}(r)&=C_{b_m}\bigg( -\beta( rsr\cdots )_{m-1} + \sum_{j=1}^{m-2} (-\beta)^j (rsr\cdots )_{m-1-j}\wp\bigg)\\
		&= -\beta(rsr\cdots)_m + \sum_{i=1}^{m-1} (-\beta)^{i} (rsr\cdots )_{m-i}\wp,
	\end{align*}}
 thus Eq. (\ref{eq:rsronr}) is true. The proof of Equations (\ref{eq:srsonr}),(\ref{eq:rsront}) and (\ref{eq:srsont}) follow similarly. The final statement follows from the fact that $C_r(\wp)=-\beta \wp=C_s(\wp)$. 
	\end{proof}
		The next theorem is key in the computations to follow. Together with Theorem \ref{thm:psimap} it allow us to compute the generators for $J_0^\Lambda$ for any finite pre-$D$-graph $\Lambda$.
		\begin{theorem}\label{thm:relations} Let $(r,s) \in S^2_{fin}$ with $r\neq s$, $m:=m_{r,s}$, and $\Gamma^{\mathcal{U}}$ be the universal pre-$D_{\mathbb{Z}}$-graph with respect to $r,s \in S$ then the relations coming from $(\rho_r\rho_s\rho_r\cdots)_{m}-(\rho_s\rho_r\rho_s\cdots)_{m}=0$ are
%		 ideal $J_0^{\Gamma^{\mathcal{U}}}\subseteq A_0[\Gamma^{\mathcal{U}}]$ is generated by $d_{m}(r,s), d_{m}(s,r)$ and $t(srs\dots)_i \wp- t(rsr\dots)_i \wp \text{ for } i =1,2, \dots, m-1.$ In other words, the relations obtained are
		\small{\[d_{m}(r,s)=0,  \quad d_{m}(s,r)=0,\quad \varnothing(srs\dots)_i \wp= \varnothing(rsr\dots)_i \wp \text{ for } i=1,2, \dots, m-1.\]}
	In other words, the ideal $J_0^{\Gamma^{\mathcal{U}}}$ is generated by the elements $d_m(r,s), d_m(s,r)$ and $\varnothing(srs\dots)_i \wp- \varnothing(rsr\dots)_i \wp$ for $ i=1,2, \dots, m-1.$
	\end{theorem}
	\begin{proof}
	By Eq. $(\ref{eq:samerhoandC})$, we have 
		\begin{equation}\label{eq:rhoC}
			(-1)^{m-1}((\rho_r \rho_s \rho_r\cdots)_{m} - (\rho_s \rho_r \rho_s\cdots)_{m})=C_{(\dots r,s,r)_{m}}- C_{(\dots s,r,s)_{m}}.
		\end{equation}
	Recall, from the way we set up notation in Eq. $(\ref{eq:relations})$, the element $\sum_qX_{q,i}^{r,s,x}q \in \mathbb{Z}[\Gamma]$ is just the coefficient of $v^i$ in either side of Eq. $(\ref{eq:rhoC})$. We now compute the right hand side of $(\ref{eq:rhoC})$ on all 4 vertices.\\
	\textit{Step 1:} By Lemmas \ref{lemma:Cconditions} and \ref{lemma:Cronr}  we get that $C_{(\dots r,s,r)_{m}}(r) - C_{(\dots s,r,s)_{m}}(r)$ equals (recall $b_m=s$ for $m$ even and $r$ for $m$ odd)
			%=\sum_{j=0}^{\lfloor \frac{m-1}{2} \rfloor} D_{m,j} \left[(\cdots C_r C_s C_r)_{m-2j}(r)- (\cdots C_s C_r C_s)_{m-2j}(r)\right]\\
				\[\sum_{j=0}^{\lfloor \frac{m-1}{2} \rfloor} D_{m,j} \left( -\beta( rsr\cdots)_{m-2j} -( rsr\cdots)_{m-2j+1} - ( rsr\cdots)_{m-2j}\wp\right),\]
 which is $	-\beta d_m(r,s)-d_m(r,s) \big((b_m \leftarrow b_{m+1})+ (b_m \leftarrow \wp)\big).$
	Thus 
	\begin{align*}
		&\sum_q X^{r,s,r}_{q,1}q =\sum_q X^{r,s,r}_{q,-1}q=-d_m(r,s)\\
		&\sum_q X^{r,s,r}_{q,0}q=-d_m(r,s) \big((b_m \leftarrow b_{m+1})+ (b_m \rightarrow \wp)\big)\\
		 &\sum_q X^{r,s,r}_{q,i}q =0 \text{ for all }i\neq 0,1,-1.
	\end{align*}
		  Therefore, the relation $d_m(r,s)=0$ is enough to generate all relations  $\sum_q X^{r,s,r}_{q,i}q=0$ for all $i \in \mathbb{Z}$. By symmetry we get that the relation $d_m(s,r)=0$ is enough to generate all relations  $\sum_q X^{r,s,s}_{q,i}q=0$ for all $i \in \mathbb{Z}$.\\
\textit{Step 2:} By Lemmas \ref{lemma:Cconditions} and \ref{lemma:Cronr} we get
		\small{\begin{align}
			&C_{(\dots r,s,r)_{m}}(\varnothing) - C_{(\dots s,r,s)_{m}}(\varnothing)
			 %= \sum_{j=0}^{\lfloor \frac{m-1}{2} \rfloor} D_{m,j} \left(  \left( \dots C_r C_s C_r\right)_{m-2j}(t) -  \left( \dots C_s C_rC_s \right)_{m-2j} (t)\right) \nonumber\\
			=\sum_{j=0}^{\lfloor \frac{m-1}{2} \rfloor} D_{m,j} \bigg(  \varnothing(rsr \dots)_{m-2j}-  \varnothing(srs \dots)_{m-2j}\nonumber \\
			&\;\;\; + \sum_{i=0}^{m-2j-2} (-\beta)^i \big( \varnothing(rsr \dots)_{m-2j-1-i} \wp  -  \varnothing(srs \dots)_{m-2j-1-i} \wp\big)\ \bigg) \label{eq:beta}
		\end{align}	}		
	For any $N\in \mathbb{Z}_{\geq1}$, the coefficient of $(-\beta^N)$ in the expression above is 
		\[ \sum_{j \in \mathbb{Z}| 0\leq N \leq m-2j-2} D_{m,j}\big( \varnothing((rsr \dots)_{m-2j-1-N}) \wp  -  \varnothing((srs \dots)_{m-2j-1-N} )\wp\big)\]
	The biggest exponent of $\beta$ in Equation (\ref{eq:beta}) is $m-2$, and thus the coefficient of $v^{m-2}$ and $v^{-(m-2)}$ (up to sign) is $(\varnothing r\wp - \varnothing s\wp)$. Hence, the relations $\sum_q X^{r,s,\varnothing}_{q, m-2}q =0$ and  $\sum_ q X^{r,s,\varnothing}_{q,-(m-2)}q=0$ are equivalent to $\varnothing r\wp-\varnothing s\wp=0$. Similarly, the coefficient of $v^{m-3}$ and $v^{-(m-3)}$ (up to sign) is $(\varnothing rs\wp - \varnothing sr\wp)$. The coefficient of $v^{m-4}$ and $v^{-(m-4)}$ (up to sign) is $(\varnothing rsr\wp - \varnothing srs\wp) + D_{m,1}(\varnothing r\wp - \varnothing s\wp)$. However, since $(\varnothing r\wp - \varnothing s\wp)$ is already one of the relations then both sets of relations $\{\sum_q X^{r,s,\varnothing}_{q, m-i}q =0\}_{i=2,3,4} $ and $\{\sum_q X^{r,s,\varnothing}_{q, -(m-i)}q =0\}_{i=2,3,4}$ equal $\{\varnothing(srs\dots)_j \wp=\varnothing(srs\dots)_j \wp\}_{j=1,2,3}$.  We continue this process to get the set of relations $\{X^{r,s,\varnothing}_{q,i}q=0\}$ for $i\in \{-(m-2), -(m-1), \dots, -1,1,\dots, m-2\}$ is equivalent to 
		\[ \{ \varnothing(srs\dots)_j \wp=\varnothing(rsr\dots)_j \wp \;\;| \;\;j =1,2, \dots, m-2\}. \]
For $N=0$ we analyse the relation $\sum_q X_{q,0}^{r,s,\varnothing}q=0$. Multiplying by the idempotents $r$, $s$, and $\wp$ on the right we obtain the relations $\varnothing(srs\dots)_{m-1} \wp=\varnothing(rsr\dots)_{m-1} \wp$ and $ \varnothing d_{m}(r,s)=0= \varnothing d_{m}(s,r)$.\\
\textit{Step 3:} it follows from the last statement of Lemma \ref{lemma:Cronr} that we get no relations from $C_{(\dots r,s,r)_{m}}(\wp) - C_{(\dots s,r,s)_{m}}(\wp)=0$. This completes the proof.
	\end{proof}

	\begin{corollary}\label{cor:finitequiver}
		Let $\Lambda$ be any pre-$D_{\mathbb{Z}}$-graph with finite vertex set and finitely many edges. For $(r,s) \in S^2_{fin}$  the relations coming from $$(\rho^{\Lambda}_r\rho^{\Lambda}_s\rho^{\Lambda}_r\cdots)_{m_{r,s}}-(\rho^{\Lambda}_s\rho^{\Lambda}_r\rho^{\Lambda}_s\cdots)_{m_{r,s}}=0$$ are \small{$\psi_{r,s} (d_{m_{r,s}}(r,s))=0$, $\psi_{r,s} (d_{m_{r,s}}(s,r))=0$}, and \[\psi_{r,s} (\varnothing(srs\dots)_i \wp- \varnothing(rsr\dots)_i \wp)=0 \text{ for } i \in \{1,2,\dots, m_{r,s}-1\}.\]
	\end{corollary}
	\begin{proof} 
		The proof follows directly from Theorems \ref{thm:psimap} and \ref{thm:relations}.
	\end{proof}

	\begin{figure}
		\begin{center}(a)
			\begin{tikzpicture}[-,auto,node distance=2cm]
				\node[draw,circle]  (1) {$S'$};
				\draw (.5,.5) node {\Large{$x$}};
			\end{tikzpicture}\qquad\qquad
			(b)
			\begin{tikzpicture}[-,auto,node distance=2cm,scale=1, every node/.style={scale=0.75}]
 				%\draw (0,0) circle (1cm);
				%\draw (-1.5,0) -- (1.5,0);
				\draw (1.25,.65) node {\Large{$x_2$}};
				\draw (-1.25,.65) node {\Large{$x_1$}};
				\draw (.5,0) node {\Large{$y$}};
				\draw (.5,.-1.5) node {\Large{$z$}};
				 \node[draw,circle]  (1) {$ r_2 $};
				  \node[draw,circle]  (2) [above right of=1] {$r_3$};
				  \node[draw,circle]  (3) [above left of=1] {$r_3$};
 				 \node [draw,circle] (4) [below of=1] {$r_1 $};
				  \path[->,every node/.style={font=\sffamily\small}]
				    (1) edge  [bend right] node {}(2)
				    	edge  [bend right]node {}(3)
					edge  [bend right]node {}(4)
				(2) edge  [bend right] node {}(1)
				(3) edge  [bend right] node {}(1)
				(4) edge  [bend right] node {}(1);
			\end{tikzpicture}
		\end{center}
	\caption{In (a) we have a one vertex pre-$D_\mathbb{Z}$-graph with associated subset $S'\subseteq S$. In (b) we have a pre-$D_\mathbb{Z}$-graph coming from a $W$-cell for $W=B_3$.}\label{fig:examples}
	\end{figure}
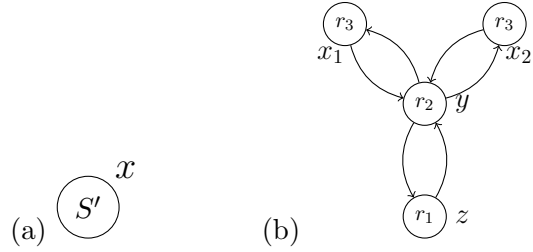
	\subsection{Examples}\label{subsec:examples} We now consider three examples of pre-$D_{\mathbb{Z}}$-graphs. First we consider a one vertex quiver, then a quiver coming from a left cell of $B_3$ and finally a special quiver which gives rise to a quotient path algebra isomorphic to an ideal of the asymptotic Hecke algebra of $(W,S)$. In all three examples we compute a generating set for the ideal $J_0^\Lambda$ using Corollary \ref{cor:finitequiver}.
	
\begin{example}[One vertex quiver]\label{subsec:onevertex}
Let $\Lambda$ be the pre-$D_{\mathbb{Z}}$-graph with one vertex $x$, no edges, and $\mathcal{L}^{\Lambda}(x)=S'\subseteq S$ (see Figure \ref{fig:examples}(a)).  Let $(r,s) \in S^2_{fin}$ with $r\neq s$ and $m:=m_{r,s}$. Since we only have one vertex then either $\psi_{r,s}(\varnothing)=0$ or $\psi_{r,s}(\wp)=0$ (or both). Thus, by Corollary \ref{cor:finitequiver}, the only relations coming from $(\rho_r\rho_s\rho_r \cdots)_m-(\rho_s\rho_r\rho_s \cdots)_m=0$ are $\psi_{r,s}(d_m(r,s))=0$ and $\psi_{r,s}(d_m(s,r))=0$. If $r,s\not \in S'$ or $r,s\in S'$ then we get no relations as $\psi_{r,s} (d_m(r,s))$ and $\psi_{r,s} (d_m(s,r))$ are already zero (because $\psi_{r,s}(r)=\psi_{r,s}(s)=0$). If $r\in S', s\not \in S'$ and $m$ is even, then we get no relations as $\psi_{r,s} (d_m(r,s))$ and $\psi_{r,s} (d_m(s,r))$ are already zero (because $\psi_{r,s}(s)=0$). If $r\in S', s\not \in S'$ and $m$ is odd, then no relation arises from $\psi_{r,s}(d_m(s,r))= 0$ (because $\psi_{r,s}(s)=0$), however $\psi_{r,s}(d_m(r,s))= x$ (up to a power of $-1$), thus we get the relation $x=0$. By symmetry we get the same relation when $r\not \in S', s\in S'$. Thus $\mathcal{R}_0=\mathbb{Z}[\Lambda]/J_0$ is the trivial algebra if there exists a pair $r,s \in S$ such that $m_{r,s}$ is odd and $r\in S', s\not\in S'$. Otherwise we have $\mathcal{R}_0\cong \mathbb{Z}[\Lambda]\cong \mathbb{Z}$. This example shows that even the most simple pre-$D_{\mathbb{Z}}$-graph can have distinct extreme behaviours for different Coxeter groups.
\end{example}
\begin{example}[Cell of $B_3$]\label{subsec:B3} Let $W=B_3, S=\{r_1,r_2,r_3\}$ (with $m_{r_1,r_2}=3, m_{r_2,r_3}=4,$ and $m_{r_1,r_3}=2$). Consider the pre-$D_{\mathbb{Z}}$-graph $\Gamma$ in Figure \ref{fig:examples}(b) (which comes from a Kazhdan-Lusztig $B_3$-cell) with vertices labelled $x_1,x_2,y,z$. Since there are no multiple edges between any two vertices we denote a path $x_1\xleftarrow{f_1}x_2\cdots\xleftarrow{f_n}x_{n+1}$ simply as $x_1x_2\cdots x_{n+1}$. We now compute the generators of $J_0^{\Gamma}$ making use of Corollary \ref{cor:finitequiver}. We split it up in three steps.\\
	\textit{Step 1:} first compute the relations coming from $\trho_{r_1}\trho_{r_2}\trho_{r_1} -\trho_{r_2}\trho_{r_1}\trho_{r_2}=0$ on the universal pre-$D$-graph with respect to $r_1,r_2$ (see Figure \ref{fig:univPDG}). By Theorem \ref{thm:relations}, these relations are 
	\begin{equation}\label{eq:relex1}
	 r_2r_1r_2- r_2=0,\ \ r_1r_2r_1-r_1=0, \ \ \varnothing r_2r_1\wp-\varnothing r_1r_2\wp=0, \ \ \varnothing r_2\wp-\varnothing r_1\wp=0.
	 \end{equation}
By Corollary \ref{cor:finitequiver} the relations on $\Gamma$ come from applying $\psi_{r_1,r_2}$ to the relations in $(\ref{eq:relex1})$. From the first two relations (from left to right) we get $yzy-y=0$ and $zyz-z=0$. From the last two relations we do not get any as $\psi(\wp)=0$ (there is no vertex with associated subset $\{r_1,r_2\}$).\\
	\textit{Step 2:} now we compute the relations coming from $\trho_{r_2}\trho_{r_3}\trho_{r_2}\trho_{r_3} -\trho_{r_3}\trho_{r_2}\trho_{r_3}\trho_{r_2}\allowbreak=0$ on the universal pre-$D$-graph with respect to $r_2,r_3$. By Theorem \ref{thm:relations}, these relations are 
	\begin{equation}\label{eq:relex2}
	 r_2r_3r_2r_3- 2r_2r_3=0,\; r_3r_2r_3r_2- 2r_3r_2=0, \; \varnothing (r_2r_3\dots)_i\wp-\varnothing (r_3r_2\dots)_i\wp=0
	 \end{equation}
	 for  $i=1,2,3.$ By Corollary \ref{cor:finitequiver} the relations on $\Gamma$ come from applying $\psi_{r_2,r_3}$ to the relations in $(\ref{eq:relex2})$ (note $\psi_{r_2,r_3}(r_2)=y$ and $\psi_{r_2,r_3}(r_3)=x_1+x_2$). From the first equation we get $yx_1yx_1+yx_1yx_2+yx_2yx_1+yx_2yx_2 - 2yx_1 - 2yx_2=0$. From the second one we get $x_1yx_1y+x_1yx_2y+x_2yx_1y+x_2yx_2y - 2x_1y - 2x_2y=0$.  From the third set of equations  we do not get any relation as $\psi(\wp)=0$.\\
	 \textit{Step 3:} finally we compute the relations coming from $\trho_{r_1}\trho_{r_3} -\trho_{r_3}\trho_{r_1}=0$ on the universal pre-$D$-graph with respect to $r_1,r_3$. By Theorem \ref{thm:relations}, these relations are $r_1r_3=0$, $r_3r_1=0$, and $tr_1p=tr_3p$. Since $\psi_{r_1,r_3} (r_1r_3)=0=\psi_{r_1,r_3} (r_3r_1)$ (there are no paths from any vertex with associated set $\{r_1\}$ to a vertex with associated set $\{r_3\}$ and viceversa) and $\psi_{r_1,r_3}(\wp)=0$, then we get no relations.\\
	 From the steps above we conclude $\mathcal{R}_0=\mathbb{Z}[\Gamma]/J_0$ where $J_0$ is the two-sided ideal generated by
	 	\begin{align*}
			\{ yzy-y, &zyz-z, yx_1yx_1+yx_2yx_1 - 2yx_1, yx_1yx_2+yx_2yx_2-2yx_2,\\
			&x_1yx_1y+x_1yx_2y - 2x_1y, x_2yx_1y+x_2yx_2y- 2x_2y\}.
		\end{align*} 
\end{example}	 
\begin{example}[Asymptotic Hecke Algebra]\label{subsec:AHA}
Let $\Gamma$ be the pre-$D_{\mathbb{Z}}$-graph where $(V, E)$ is the complete quiver on the vertex set $S$,  and for $s \in V$, $\mathcal{L}(s)=\{s\}$. Let $(r,s) \in S^2_{fin}$ with $r\neq s$, and let $\Gamma^\mathcal{U}$ be the universal quiver with respect to $r,s$. To avoid confusion we use the original notation for the vertices in $\Gamma^\mathcal{U}$ (i.e, $V^{\Gamma^\mathcal{U}}=\{x_r,x_s,x_{\varnothing},x_\wp\}$, see Figure \ref{fig:univPDG}). Then $\psi_{r,s}(x_r)=r$, $\psi_{r,s}(x_s)=s$, $\psi_{r,s}(x_{\varnothing})=\sum_{t \in S\setminus \{r,s\}} t$ and $\psi_{r,s}(x_\wp) = 0$. 

By the above and Corollary \ref{cor:finitequiver} we conclude that the relations coming from $(\rho_r\rho_s\rho_s \cdots)_{m_{r,s}}- (\rho_s\rho_r\rho_s \cdots)_{m_{r,s}}=0$ are 
	\[ \psi_{r,s}(d_{m_{r,s}}(x_r,x_s))=0, \quad \psi_{r,s}(d_{m_{r,s}}(x_s,x_r))=0.\]
Similar to $c_m(r,s) \in \mathbb{Z}[\Gamma^{\mathcal{U}}]$, lets define
\begin{equation*}
 \mathfrak{c}_m(r,s) =\cm{m}{r}{s} \in \mathbb{Z}[\Gamma],
	\end{equation*}
Let $\mathfrak{d}_m(r,s)=\mathfrak{c}_m(r,s)$ if $m$ is even and $\mathfrak{d}_m(r,s)=\mathfrak{c}_m(s,r)$ if $m$ is odd. Then
the relations defining $J_0^\Gamma$ are $\mathfrak{d}_{m_{r,s}}(r,s)=0$ and $\mathfrak{d}_{m_{r,s}}(s,r)=0$. It follows then that $\mathcal{R}_0^\Gamma = \mathbb{Z}[\Gamma]/ J_0$ where $J_0$ is generated by the elements $\mathfrak{c}_m(r,s)$ for all $(r,s) \in S^2_{fin}$ with $r\neq s$. It is known by Matthew Dyer \cite{Dyer1} that the ring $\mathcal{R}_0^\Gamma$ is isomorphic to the ideal $J^c$ of the asymptotic Hecke-algebra associated with the union of two-sided cells with $a$-value 1 (see \cite{L1}, Chapter 18). This theory provides a natural construction of $J^c$ when defined (the asymptotic Hecke algebra has been proved to exist in numerous cases but it is still conjectural for a general Coxeter group).
\end{example}
%%%%%%%%%%%%%
%%%%%%%%%%%%%
%%%%%%%%%%%%%

\section{$D$-graphs over non-commutative algebras}\label{sec:WGONCA}
\subsection{Introduction}
	In this section we show how the main objects of study in this paper can be seen as examples of ``$D$-graphs over non-commutative algebras".
	Let $(D=(W,S,A,\allowbreak (a_r)_{r\in S},(b_r)_{r\in S}),A_0)$ be an extended Hecke datum.
	\begin{definition}\label{def:Dgraphovernca}
		A $D$-graph over a (posibly non-commutative) $A$-algebra $R$ is a pre-$D$-graph $\Gamma$ together with a map $\mu:E^{\Gamma} \rightarrow R$ such that
		\begin{enumerate}[(i)]
			\item $R$ is an $A$-algebra with a set of enough orthogonal idempotents \small{$\{e_x\}_{x\in V^{\Gamma}}$}
			\item $\mu^{\Gamma} (x\xleftarrow{f} y) \in e_x Re_y$ for all $f\in E^{\Gamma}$
			\item  \label{property:taubraid}for all $(r, s) \in S^2_{fin}$ we have 
			\[ (\tau^\Gamma_r \tau^\Gamma_s \tau^\Gamma_r \dots)_{m_{r,s}} - ( \tau^\Gamma_s \tau^\Gamma_r \tau^\Gamma_s \cdots)_{m_{r,s}} = 0 \in End_R(R),\] 
		where 
			\begin{equation*}\label{eq:taugeneral}
				 \tau_r^\Gamma (e_x) = \begin{cases}
				a_r e_x + \sum\limits_{\substack{y,f|x \xleftarrow{f} y \\r \in \mathcal{L}(y)}}\mu(f)e_y & \text{ if $r \not\in \mathcal{L}(x)$}\\
				-b_r e_x & \text{ if $r \in \mathcal{L}(x).$}\end{cases}
			\end{equation*}
		\end{enumerate}
	\end{definition}  
	\begin{proposition}\label{prop:genrep}
		Given a $D$-graph $\Gamma$ over an $A$-algebra $R$ we get a representation of the Hecke algebra $\mathcal{H}(D)$ given by $\tau^\Gamma:\mathcal{H}(D) \to End_A(R)$ such that $\tau^\Gamma(T_r) = \tau^\Gamma_r$.
	\end{proposition}
	\begin{proof}
		This proof follows as in the proof of Proposition \ref{prop:tauwd}.
	\end{proof}
	\subsection{Examples}
	
\begin{example}[Quotient path algebra]\label{ex:qpa}Given an extended Hecke datum $(D,A_0)$ and a pre-$D$-graph $\Gamma$ as in Definition \ref{def:preDgraph} we obtain a $D$-graph over (the quotient path algebra) $\mathcal{R}^\Gamma$, where $\mathcal{R}^\Gamma$ is an $A$-algebra with enough orthogonal idempotents $\{[x]| x\in V^{\Gamma}\}$ and map $\mu^{\Gamma}$ defined as $\mu^{\Gamma}(x\xleftarrow{f}y) := [x\xleftarrow{f}y] \in \mathcal{R}^\Gamma$ for all $x\xleftarrow{f}y \in E^{\Gamma}$.  By the way we constructed $\mathcal{R}$, property $(\ref{property:taubraid})$ in Definition \ref{def:Dgraphovernca} is guaranteed to be satisfied. In this specific example, the representation we obtain from Proposition \ref{prop:genrep} is denoted as $\trho^{\Gamma}$ (as it has been denoted throughout this document). In Section \ref{sec:univ} we discuss a ``universality" property of $\trho^{\Gamma}$.
\end{example}
	\begin{example}[Kazhdan-Lusztig module]
Given a $D$-graph $\Gamma$ (in the original sense) as in Definition \ref{def:preDgraph}, we use the Morita equivalence result to construct the $A$-module (in fact, $A$-algebra) $A^{V,V}$ as established in Section \ref{subsec:morita}. It is not hard to see that the set $\{e_{x,x}|x\in V^{\Gamma}\}$ is a set of enough orthogonal idempotents for $A^{V,V}$. Consider the map $\mu^{\Gamma}(x\xleftarrow{f}y)=\mu(e)e_{x,y} \in e_{x,x}A^{V,V}e_{y,y}$ (where $\mu$ is the map coming from the $D$-graph). Under these conditions $\Gamma$ is a $D$-graph over the $A$-algebra $A^{V,V}$. 
\end{example}
\subsection{Universality of $\trho^{\Gamma}$}\label{sec:univ}
		The representation $\trho^{\Gamma}$ is universal in the following sense: let $\Gamma$ be any $D$-graph over an $A$-algebra $R$, and suppose $R$ is a free $A$-module (e.g. $A^{V,V}$), and $\tau^\Gamma$ be the representation of $\mathcal{H}(D)$ described in Proposition \ref{prop:genrep}, then $\tau^\Gamma$ is obtained as a ``specialization" of the representation $\trho^{\Gamma}$ given by the $D$-graph $\Gamma$ over the quotient path algebra $\mathcal{R}^\Gamma$ (as described in Example \ref{ex:qpa}). We make the statement precise in the following theorem. 
	\begin{theorem}\label{thm:universality}
		Let $\Gamma$ be any $D$-graph over an $A$-algebra $R=\oplus_{x,y \in V} e_xRe_y$, and assume $R$ is a free $A$-module. Let $\tau^\Gamma$ be the representation of $\mathcal{H}(D)$ described in Proposition \ref{prop:genrep}, and $\trho^{\Gamma}$ be the representation of $\mathcal{H}(D)$ described in Example \ref{ex:qpa} on the quotient path algebra $\mathcal{R}^\Gamma$. Let $\Theta$ be the map from $A[\Gamma]$ to $R$ that sends 
	\[x\in V^\Gamma \mapsto e_x \text{\ and \ } [x\xleftarrow{e}y] \mapsto \mu(e) \in e_xRe_y.\]
			Then 
			\begin{enumerate}[(I)]
				\item $\Theta$ is an $A$-algebra homomorphism. \label{property:homo}
				\item For $r_1, r_2, \dots, r_n$ we get 
				\[(\tau_{r_1}^\Gamma\tau_{r_2}^\Gamma\cdots \tau_{r_n}^\Gamma)\circ \Theta = \Theta\circ(\rho_{r_1}^\Gamma\rho_{r_2}^\Gamma\cdots \rho_{r_n}^\Gamma).\]\label{property:commute}
				\item The map $\Theta$ factors through $\mathcal{R}^\Gamma$ i.e., there is a map $\tTheta: \mathcal{R}^\Gamma \to R$ such that $\tTheta([p]) = \Theta(p)$ for all $p \in A[\Gamma]$.\label{property:factors}
				\item For any $r \in S$, $\tau_r^\Gamma(e_x) = \tTheta(\trho^{\Gamma}_r(x))$. \label{property:spec}
				\item $\tTheta$ is an equivariant map (i.e., $\tau^\Gamma\circ \tTheta = \tTheta\circ\trho^\Gamma.$)\label{property:Tequiv}
			\end{enumerate}
	\end{theorem}
	%%%
	%%%
\begin{proof}
The proof follows similar to those in Theorem \ref{thm:eqmapmorita} and Lemma \ref{lemma:utaumorita}.\end{proof}

%%%%%%%
%%%%%%%
%%%%%%%
\section*{Acknowledgments}
The author would like to thank Matthew Dyer for numerous enlightening conversations on these topics. 

\end{document}